\documentclass{amsart}
\usepackage{indentfirst,enumerate,mathrsfs}
\usepackage{array}
\usepackage[pdftex]{hyperref}
\def\NN{\mathbb N}

\def\DD{\mathcal D}
\def\RR{\mathbb R}
\def\PP{\mathcal P}

\def\EE{\mathbb E}

\def\HH{\mathcal H}

\def\ZZ{\mathbb Z}
\def\xx{\mathbf{x}}
\def\yy{\mathbf{y}}
\def\zz{\mathbf{z}}

\newcommand{\la} {\lambda}
\newcommand{\al} {\alpha}
\newcommand{\Ran}{\mathcal{R}}
\newcommand{\LL}{\mathscr{L}^2(X,\rho_X;Y)}
\newcommand{\LN}{\mathscr{L}^2(X,\nu;Y)}

\newcommand{\fbar}{\bar f}
\newcommand{\argmin}{\operatornamewithlimits{argmin}}
\newcommand{\paren}[1]{\left(#1\right)}
\newcommand{\brac}[1]{\left\{#1\right\}}
\newcommand{\sbrac}[1]{\left[#1\right]}
\newcommand{\inner}[1]{\left\langle#1\right\rangle}
\newcommand{\norm}[1]{\left\|#1\right\|}

\newcommand{\abs}[1]{\left\lvert #1 \right\rvert}
\newcommand{\ol}[1]{\overline{#1}}
\newcommand{\tr}{\operatorname{tr}}
\newtheorem{theorem}{Theorem}[section]

\newtheorem{corollary}[theorem]{Corollary}
\newtheorem{definition}[theorem]{Definition}
\newtheorem{proposition}[theorem]{Proposition}
\newtheorem{assumption}{Assumption}
\newtheorem{remark}[theorem]{Remark}
\newtheorem{example}[theorem]{Example}
\RequirePackage{color}
\newcommand{\gillesb}[1]{\textcolor{blue}{#1}}
\allowdisplaybreaks
\begin{document}
\title[Tikhonov regularization for non-linear inverse learning problems]{Convergence analysis of Tikhonov regularization for non-linear statistical inverse learning problems} 
\author{Abhishake}
\address{Institute of Mathematics, University of Potsdam,
  Karl-Liebknecht-Strasse 24-25, 14476 Potsdam, Germany}
\email{abhishake@uni-potsdam.de}

\author{Gilles Blanchard}
\address{Institute of Mathematics, University of Potsdam,
  Karl-Liebknecht-Strasse 24-25, 14476 Potsdam, Germany}
\email{gilles.blanchard@math.uni-potsdam.de}
 
\author{Peter Math{\'e}}
\address{Weierstrass Institute for Applied Analysis and Stochastics,
  Mohrenstrasse 39, 10117 Berlin, Germany} 
\email{peter.mathe@wias-berlin.de}

\date{}
\thanks{This research has been partially funded by Deutsche Forschungsgemeinschaft (DFG) through grant CRC 1294 ``Data Assimilation", Project (A04) ``Non-linear statistical inverse problems with random observations".} 

\begin{abstract}
We study a non-linear statistical inverse learning problem, where we observe the noisy image of a quantity through a non-linear operator at some random design points. We consider the widely used Tikhonov regularization (or method of regularization, MOR) approach to reconstruct the estimator of the quantity for the non-linear ill-posed inverse problem. The estimator is defined as the minimizer of a Tikhonov functional, which is the sum of a data misfit term and a quadratic penalty term. We develop a theoretical analysis for the minimizer of the Tikhonov regularization scheme using the ansatz of reproducing kernel Hilbert spaces. We discuss optimal rates of convergence for the proposed scheme, uniformly over classes of admissible solutions, defined through appropriate source conditions. 
\end{abstract} 
\maketitle

\section{Introduction}\label{Introduction}

In this study, we shall consider non-linear operator equations of the form
\begin{equation}
  \label{eq:base}
  A(f) = g,
\end{equation}
where the non-linear mapping~$A\colon \DD(A) \subseteq \HH_{1} \to \HH_{2}$ is acting between the real separable Hilbert spaces~$\HH_{1}$ and~$\HH_{2}$. Such non-linear inverse problems occur in many situations, and examples are given in the seminal monograph~\cite{Engl}. Of special importance are problems of parameter identification in partial differential equations, and we mention the monograph~\cite[Chapt.~1]{Isakov}, and the more recent~\cite{Schuster}.

Within the \emph{classical setup} it is assumed to observe noisy data~$g^{\delta}\in\HH_{2}$ with~$\norm{g^{\delta} - A(f)}_{\HH_{2}}\leq \delta$, where the number~$\delta>0$ denotes the noise level.  In \emph{supervised learning}, it is assumed that the image space~$\HH_{2}$ consists of functions, given on some domain~$X$ and taking values in another Hilbert space~$Y$. Moreover, function evaluation is continuous, such that for~$x\in X$ the values~$g(x)= A(f)(x)$  are well defined elements in~$Y$. The goal is to \emph{learn} the unknown and indirectly observed quantity~$f\in\HH_{1}$ from examples, given in the form of i.i.d. samples~$\zz=\{(x_i,y_i)\}_{i=1}^m\in(X\times Y)^m$, where the elements~$y_{i},\ i=1,\dots,m$ are noisy observations of~$g(x_{i}),\ i=1,\dots,m$ at random points~$x_{i},\ i=1,\dots,m$ of the form
\begin{equation}\label{Model}
  y_i:= g(x_i) + \varepsilon_i \quad\text{for  }  i = 1,\ldots,m,\quad \text{where}\ g= A(f).
\end{equation}
We assume that the random observations of~$\zz$ are drawn independently and identically according to some unknown joint probability distribution~$\rho$ on the sample space~$Z=X\times Y$.  The noise terms~$(\varepsilon_i)_{i=1}^m$ are independent centered random variables satisfying~$\mathbb{E}_{y_i}[\varepsilon_i |x_i] = 0$. The cardinality~$m$ of the samples~$\zz$ is called sample size.

In the case of random observations, the literature is much more scarce than for the classical setup. Milestone work includes \cite{OSullivan} which considers asymptotic analysis for the generalized Tikhonov regularization for~(\ref{fzl}) using the linearization technique. The reference \cite{Bissantz} considers a 2-step approach, however, it is assumed that the norm in~$\LN$ (the space of square integrable functions with respected the probability measure~$\nu$ on~$X$) is observable, an unrealistic assumption if the only information on~$\nu$ is available through the points~$(x_1,\ldots,x_m)$. The references \cite{Bauer09} and \cite{Hohage} consider respectively a Gauss-Newton algorithm and the MOR method for certain non-linear inverse problem, but also in the idealized setting of Hilbertian white or colored noise, which can only cover sampling effects when~$\LN$ is known. Loubes et al. \cite{Loubes} consider~(\ref{fzl}) under a fixed design and concentrate on the problem of model selection. Finally, the recent work \cite{Ray} analyzes rates of convergence in a model where observations are of the form~$h(Kf)(x)$ perturbed by noise, but only in a white noise model and for specific, uni-variate non-linear link functions~$h$, linear operator~$K$.

A widely used approach to stabilizing the estimation problem~(\ref{Model}) is Tikhonov regularization or regularized least-squares algorithm or method of regularization (MOR). The estimate of the true solution of \eqref{Model} is obtained by minimizing an objective function consisting of an error term measuring the fit to the data plus a smoothness term measuring the complexity of the quantity~$f$. For the non-linear statistical inverse learning problem~\eqref{Model}, the regularization scheme over the hypothesis space~$\HH_1$ can be described as
\begin{equation}\label{fzl}
  f_{\zz,\la} =
  \argmin_{f\in\mathcal{D}(A)\subset\HH_1}\left\{\frac{1}{m}\sum\limits_{i=1}^m\norm{A(f)(x_i)-y_i}_Y^2+\la 
    \norm{f-\fbar}_{\HH_1}^2\right\}. 
\end{equation}
Here~$\fbar\in\HH_1$ denotes some initial guess of the true solution, which offers the possibility to incorporate {\em a priori} information. The regularization parameter~$\la$ is positive which controls the trade-off between the error term measuring the fitness of data and the complexity of the solution measured in the norm in~$\HH_1$.

The  objective of this paper is to analyze the theoretical properties of the regularized least-squares estimator~$f_{\zz,\la}$, in particular, the asymptotic performance of the algorithm is evaluated by the bounds and the rates of convergence of the regularized least-squares estimator~$f_{\zz,\la}$ in the reproducing kernel ansatz. Precisely, we develop a non-asymptotic analysis of Tikhonov regularization~(\ref{fzl}) for the non-linear statistical inverse learning problem based on the tools that have been developed for the modern mathematical study of reproducing kernel methods. The challenges specific to the studied problem are that the considered model is an inverse problem (rather than a pure prediction problem) and non-linear. The upper rate of convergence for the regularized least-squares estimator~$f_{\zz,\la}$ to the true solution is described in probabilistic sense by exponential tail inequalities. For sample size~$m$, a positive decreasing function~$m\mapsto \varepsilon(m)$, and for confidence level~$0<\eta<1$, we establish bounds of the form
$$
\mathbb{P}_{\zz\in Z^m}\left\{\norm{f_{\zz,\la}-f}_{\HH_1}\leq
  \varepsilon(m)\log\left(\frac{1}{\eta}\right)\right\}\geq 1-\eta.
$$
The function~$m\mapsto \varepsilon(m)$ describes the rate of convergence as~$m\to 0$. The upper rate of convergence is complemented by a minimax lower bound for any learning algorithm for considered non-linear statistical inverse problem. The lower rate result shows that the error rate attained by Tikhonov regularization scheme for suitable parameter choice of the regularization parameter is optimal on a suitable class of probability measures.

Now we review previous results concerning regularization algorithms on different learning schemes which are directly comparable to our results: Rastogi et al. \cite{Rastogi} and Blanchard et al. \cite{Blanchard}. For convenience, we tried to present the most essential points in a unified way in Table~\ref{comparision}.
\begin{table}[ht!]
  \centering
    \begin{tabular}{|m{1.4cm}|l|c|m{4cm}|m{1.4cm}|m{1.3cm}|}
      \hline
      &~$\|f_{\zz,\la}\hspace{-.1cm}-\hspace{-.1cm}f\|_{\HH_1}$ &  \small{Smoothness} & \small{  ~~~~Scheme} & \small{general source condition} & \small{Optimal rates} \\ 
      \hline
      \small{Rastogi et al. \cite{Rastogi}}
      &~$m^{-\frac{br}{2br+b+1}}$  &~$0\leq r\leq 1$ & \small{General regularization for direct learning} &$~~\surd$&$~~\surd$\\
      \hline
      \small{Blanchard et al. \cite{Blanchard}}  &~$m^{-\frac{br}{2br+b+1}}$  &~$0\leq r\leq 1$ & \small{General regularization for linear inverse learning} &&$~~\surd$\\
      \hline
      \small{Our Results} &~$m^{-\frac{br}{2br+b+1}}$ &~$\frac{1}{2}\leq r\leq 1$ & \small{Tikhonov regularization for non-linear inverse learning} &$~~\surd$&$~~\surd$\\
      \hline
    \end{tabular}
  \caption{Convergence rates of the regularized least-squares algorithms on different learning schemes}\label{comparision}
\end{table}

In this table, the parameter~$r$ corresponds to a (H\"older type) smoothness assumption for the unknown true solution, and the parameter~$b>1$ corresponds to the decay rate of the eigenvalues of the covariance operator, both to be introduced below in Assumption~\ref{source.cond}, and Assumption~\ref{poly.decay}, respectively.

The model~(\ref{Model}) covers non-parametric regression under random design (which we also call the direct problem, i.e.,~$A = I$),  and the linear statistical inverse learning problem. Thus, introducing a general non-linear operator~$A$ gives a unified approach to the different learning problems.  In the direct learning setting, Rastogi et al. \cite{Rastogi} obtained minimax optimal rates of convergence for general regularization under general source condition. Blanchard et al.~\cite{Blanchard} considered the general regularization for the linear statistical inverse learning problem. They generalized the convergence analysis of the direct learning scheme to the inverse learning setting and achieved the minimax optimal rates of convergence for general regularization under a H{\"o}lder source condition. They considered that the image of the operator~$A$ is a reproducing kernel Hilbert space which is a special case of our general assumption that~$Im(A)$ is contained in a reproducing kernel Hilbert space. Here, we consider Tikhonov regularization for the non-linear statistical inverse learning problem. We obtain minimax optimal rates of convergence under a general source condition. The assumptions on the non-linear operator~$A$ (see~Assumption~\ref{A.assumption}, and the condition~\eqref{smallness}, below) allow us to estimate the error bounds for the source condition under some additional constraint, which for H\"older source condition ($\phi(t)=t^r$) corresponds to the range~$\frac{1}{2}\leq r\leq 1$.

The structure of the paper is as follows. In Section~\ref{Definitions}, we introduce the  basic setup and notation for supervised learning problems in a reproducing kernel Hilbert space framework. In Sections~\ref{sec:consistency} and~\ref{sec:convergence}, we discuss the  main results of this paper on consistency and error bounds of the regularized least-squares solution~$f_{\zz,\la}$ under certain assumptions on the (unknown) joint probability measure~$\rho$, and on the (non-linear) mapping~$A$. We establish  minimax rates of convergence over the regularity classes defined through appropriate source conditions by using the concept of effective dimension. In Section~\ref{Discussion}, we present a concluding discussion on some further aspects of the results. In the appendix, we establish the concentration inequalities, perturbation results and the proofs of consistency results, upper error bounds and lower error bounds.

\section{Setup and basic definitions}\label{Definitions}
In this section, we discuss the mathematical concepts and definitions used in our analysis. We start with a brief description of the reproducing kernel Hilbert spaces since our approximation schemes will be built in such spaces. The vector-valued reproducing kernel Hilbert spaces are the extension of real-valued reproducing kernel Hilbert spaces, see e.g.~\cite{Micchelli1}.

\begin{definition}
Let~$X$ be a non-empty set,~$(Y,\langle\cdot,\cdot\rangle_Y)$ be a  real separable Hilbert space and~$\HH$ be a Hilbert space~$\HH$ of functions from~$X$ to~$Y$. If the linear functional~$F_{x,y}:\HH \to \RR$, defined by
$$F_{x,y}(f)=\langle y,f(x)\rangle_Y \qquad \forall f \in \HH,$$
is continuous for every~$x \in X$ and~$y\in Y$, then~$\HH$ is called vector-valued reproducing kernel Hilbert space.
\end{definition}

For the Banach space of bounded linear operators~$\mathcal{L}(Y):Y\to Y$, a function~$K:X\times X\to \mathcal{L}(Y)$ is said to be an operator-valued positive semi-definite kernel if for each pair~$(x,z)\in X\times X$,~$K(x,z)^*=K(z,x)$, and for every finite set of points~$\{x_i\}_{i=1}^N\subset X$ and~$\{y_i\}_{i=1}^N\subset Y$,
\[\sum_{i,j=1}^N\langle y_i,K(x_i,x_j)y_j\rangle_Y\geq 0.\]

For every operator-valued positive semi-definite kernel,~$K:X \times X \to \mathcal{L}(Y)$, there exists a unique vector-valued reproducing kernel Hilbert space~$(\HH,\langle\cdot,\cdot\rangle_{\HH})$ of functions from ~$X$ to~$Y$ satisfying the following conditions:
\begin{enumerate}[(i)]
\item For all~$x\in X$ and~$y\in Y$, the function~$K_xy=K(\cdot,x)y$, defined by
  \[z \in X \mapsto (K_xy)(z)=K(z,x)y \in Y,\] belongs to~$\HH$; this allows us to define the linear mapping~$K_x: Y \rightarrow \HH: y \mapsto K_xy$.
\item The span of the set~$\{K_xy:x\in X, y\in Y\}$ is dense in~$\HH$.
\item For all~$f\in \HH$,~$x\in X$ and~$y \in Y$,~$\langle f(x),y\rangle_Y=\langle f,K_xy\rangle_{\HH}$, in other words~$f(x) = K_x^* f$ (reproducing property).
\end{enumerate}

Moreover, there is a one-to-one correspondence between operator-valued positive semi-definite kernels and vector-valued reproducing kernel Hilbert spaces \cite{Micchelli1}. In special case, when~$Y$ is a bounded subset of~$\RR$, the reproducing kernel Hilbert space is said to be real-valued reproducing kernel Hilbert space. In this case, the operator-valued positive semi-definite kernel becomes the symmetric, positive semi-definite kernel~$K:X \times X \to \RR$ and each reproducing kernel Hilbert space~$\HH$ can described as the completion of the span of the set~$\{K_x \in \HH : x \in X \}$ for~$K_x:X \to \RR: t\mapsto K_x(t)=K(x,t)$. Moreover, for every function~$f$ in the reproducing kernel Hilbert space~$\HH$, the reproducing property can be described as~$f(x)=\langle f,K_x\rangle_{\HH}$.  

First, we assume that the input space~$X$ be a Polish space and the output space~$(Y, \langle \cdot,\cdot\rangle_Y)$ be a real separable Hilbert space. Hence, the joint probability measure~$\rho$ on the sample space~$Z=X\times Y$ can be described as~$\rho(x, y) = \rho(y|x)\rho_X(x)$, where~$\rho(y|x)$ is the conditional distribution of~$y$ given~$x$ and~$\rho_X$ is the marginal distribution on~$X$.

We specify the abstract framework for the present study.  We consider that random observations~$\{(x_i,y_i)\}_{i=1}^m$ follow the model~$y= A(f)(x)+\varepsilon$ with the centered noise~$\varepsilon$.

\begin{assumption}[True solution~$f_\rho$] \label{assmpt2} The conditional expectation w.r.t.~$\rho$ of~$y$ given~$x$ exists (a.s.), and there exists~$f_\rho \in \mathrm{int}(\mathcal{D}(A))\subset\HH_1~$ such that
  \begin{equation*}\label{fp}
    \mathbb{E}_y[y |x] =\int_Y y d\rho(y|x)= A(f_\rho)(x), \text{ for all } x\in X.
  \end{equation*}
\end{assumption}
The element~$f_\rho$ is the true solution which we aim at estimating.

\begin{assumption}[Noise condition]\label{noise.cond}
  There exist some constants~$M,\Sigma$ such that for almost all~$x\in X$,
    \begin{equation*}
      \int_Y\left(e^{\norm{y-A(f_\rho)(x)}_Y/M}-\frac{\norm{y-A(f_\rho)(x)}_Y}{M}-1\right)d\rho(y|x)\leq\frac{\Sigma^2}{2M^2}.
    \end{equation*}
  \end{assumption}
This Assumption is usually referred to as a \emph{Bernstein-type assumption}.
  
Concerning the Hilbert space~$\HH_{2}$, we assume the following throughout the paper.
\begin{assumption}[Vector valued reproducing kernel Hilbert space~$\HH_{2}$] \label{assmpt1} 
We assume~$\HH_2$ to be a vector-valued reproducing kernel Hilbert space of functions~$f:X\to Y$ corresponding to the kernel~$K:X\times X\to \mathcal{L}(Y)$ such that
  \begin{enumerate}[(i)]
  \item For all~$x\in X$,~$K_x:Y\to\HH_2$ is a Hilbert-Schmidt
    operator, and
    \[\kappa^2:=\sup_{x \in X} \norm{K_x}^2_{HS} = {\sup_{x \in
          X}\tr(K_x^*K_x)}<\infty,\] implying in particular
    that~$\HH_2\subset \LL$.
  \item The real-valued function~$\varsigma:X\times X \to \RR$, defined by~$\varsigma(x,t)=\langle K_tv,K_xw\rangle_{\HH_2}$, is measurable~$\forall v,w\in Y$.
  \end{enumerate}
\end{assumption}

Note that in case of real-valued functions ($Y\subset\RR$), Assumption~\ref{assmpt1} simplifies to the condition that the kernel is measurable and~$\kappa^2:=\sup_{x \in X} \norm{K_x}^2_{\HH_2}=\sup_{x \in X}K(x,x)<\infty$.

The operator~$I_K$ denotes the canonical injection map~$\HH_2 \to \LL$, that
\[\norm{I_Kf}^2_{\LL}=\int_X\norm{f(x)}_Y^2d\rho_X(x)
  =\int_X\norm{K_x^*f}_{Y}^2d\rho_X(x) \leq
  \kappa^2\norm{f}^2_{\HH_1}.
\]
We denote~$L_{K}:= I_K^{\ast}I_K\colon \HH_{2}\to \HH_{2}$ the corresponding covariance operator.

\section{Consistency}
\label{sec:consistency}


We establish consistency in RMS sense and almost surely of  Tikhonov regularization in the sense that~$\|f_{\zz,\la}-f_\rho\|_{\HH_1}$ as~$|\zz|=m\to\infty$. For this we need weak assumptions on the operator.
\begin{assumption}
  [Lispschitz continuity]\label{ass:Lipschitz}
  We suppose that~$\mathcal{D}(A)$ is weakly closed with nonempty interior and~$A:\mathcal{D}(A)\subset\HH_1\to \HH_2$ is Lipschitz continuous, one-to-one.
\end{assumption}

The inequality~$\norm{I_Kg}_{\LL}\leq\kappa\norm{g}_{\HH_2}$ for~$g\in\HH_2$ and the continuity of the operator~$A:\HH_1\to\HH_2$ implies that~$I_KA:\HH_1\to\HH_2\hookrightarrow\LL$ is also continuous. Since~$\mathcal{D}(A)$ is weakly closed, therefore~$I_KA$ is weakly sequentially closed\footnote{i.e., if a sequence~$(f_m)_{m\in\NN} \subset \mathcal{D}(A)$ converges weakly to some~$f \in \HH_1$ and if the sequence~$(A(f_m))_{m\in\NN}\subset \LL$ converges weakly to some~$g\in \LL$, then~$f\in\mathcal{D}(A)$ and~$A(f)=g$.}. For the continuous and weakly sequentially closed opeator~$A$, there exists a global minimizer of the functional in~(\ref{fzl}). But it is not necessarily unique since~$A$ is non-linear (see \cite[Section~4.1.1]{Schuster}). 

The proofs of Theorems~\ref{thm:consistency-rms},~\ref{thm:asconsistency} will be given in Appendix~\ref{sec:proof-consistency}.

\begin{theorem}\label{thm:consistency-rms}
Suppose that Assumptions~\ref{assmpt2}, \ref{assmpt1}, \ref{ass:Lipschitz} hold true and~$\sigma_\rho^2:=\int_Z\norm{y-A(f_\rho)(x)}_Y^2d\rho(x,y)<\infty$. Let~$f_{\zz,\la}$ denote a (not necessarily unique) solution to the minimization problem~(\ref{fzl}) and assume that the regularization parameter~$\la(m)>0$ is chosen such that
\begin{equation}\label{la.choice1}
\la\to 0,~~~\frac{1}{\la\sqrt{m}}\to 0 \text{   as   } m\to\infty.
\end{equation}
Then we have that
\begin{equation}\label{consistency.rate.exp}
\mathbb{E}_\zz\paren{\|f_{\zz,\la}-f_\rho\|_{\HH_1}^2}\to 0   \text{   as   } \abs{\zz}=m\to\infty.
\end{equation}
\end{theorem}

\begin{remark}\label{rem:2ndmoment}
As can be seen from the proof, the existence of arbitrary  moments, as required in Assumption~\ref{noise.cond} is not needed. Instead, only the existence of second moments is used, as seen from the introduction of~$\sigma_{\rho}$.
\end{remark}
The previous result can be strengthened as follows.
\begin{theorem}\label{thm:asconsistency}
Suppose that Assumptions~\ref{assmpt2}--\ref{ass:Lipschitz} hold true. Let~$f_{\zz,\la}$ denote a (not necessarily unique) solution to the minimization problem~(\ref{fzl}) and assume that the regularization parameter~$\la(m)>0$ is chosen such that
\begin{equation}\label{la.choice}
\la\to 0,~~~\frac{\log m}{\la\sqrt{m}}\to 0 \text{   as   } m\to\infty.
\end{equation}
Then we have that
\begin{equation}\label{consistency.rate}
\|f_{\zz,\la}-f_\rho\|_{\HH_1}\to 0   \text{  almost surely as   } m\to\infty.
\end{equation}
\end{theorem}

\section{Convergence rates}
\label{sec:convergence}

In order to derive rates of convergence additional assumptions are made on the operator~$A$. We need to introduce the corresponding notion of smoothness of the true solution~$f_\rho$ from Assumption~\ref{assmpt2}. We discuss the class of probability measures defined through the appropriate source condition which describe the smoothness of the true solution.

Following the work of Engl et al. \cite[Chapt.~10]{Engl} on `classical' non-linear inverse problems, we consider the following assumption:

\begin{assumption}[Non-linearity of the operator]\label{A.assumption}
  We assume that~$\mathcal{D}(A)$ is convex with nonempty interior,~$A:\mathcal{D}(A)\subset\HH_1\to \HH_2\hookrightarrow\LL$ is weakly sequentially closed and one-to-one. Furthermore, we assume that
  \begin{enumerate}[(i)]
  \item~$A$ is Fr\'{e}chet differentiable, 
  \item the Fr\'{e}chet derivative~$A'(f)$ of~$A$ at~$f$ is bounded in a sufficiently large ball~$\mathcal{B}_d(f_\rho)$, i.e., there exists~$L < \infty$ such that
   ~$$
    \norm{A'(f)}_{\HH_1\to\HH_2}\leq L \qquad \forall~f\in
    \mathcal{B}_d(f_\rho)\cap\mathcal{D}(A)\subset\HH_1,
   ~$$
    and
  \item \label{it:gamma}there exists~$\gamma\geq 0$ such that for
    all~$f\in \mathcal{B}_d(f_\rho)\cap\mathcal{D}(A)\subset\HH_1$ we have,
$$\norm{I_K\brac{A(f)-A(f_\rho)-A'(f_\rho)(f-f_\rho)}}_{\LL} \leq \frac{\gamma}{2}\norm{f-f_\rho}_{\HH_1}^2.$$  
\end{enumerate}
\end{assumption}
\begin{remark}
The condition~(\ref{it:gamma}) also holds true under the stronger assumption that~$A'$ is Lipschitz for the operator norm (see \cite[Chapt.~10]{Engl}), i.e.,\ 
$$\norm{I_K\brac{A'(f)-A'(f_\rho)}}_{\HH_1\to\LL}\leq \gamma\norm{f-f_\rho}_{\HH_1}.$$A sufficient condition for weak sequential closedness is that~$\mathcal{D}(A)$ is weakly closed (e.g. closed and convex) and~$A$ is weakly continuous. Note that under the Fr{\'e}chet differentiability of~$A:\DD(A)\subset\HH_1\to\HH_2$ (Assumption~\ref{A.assumption} (ii)), the operator~$A$ is Lipschitz continuous with Lipschitz constant~$L$. 
\end{remark}

To illustrate the general setting, we consider a family of integral operators on the Sobolev space satisfying the above assumptions, where the kernel~$K$ is completely explicit.

\begin{example}
  Let~$\HH_1=\HH_2=\HH$ be the Sobolev space~$W^{k,2}(\RR^d)$ of differential order~$k$ (based on $\mathscr{L}^2(\RR^d,\rho_X;\RR)$), for the integer~$k>\frac{d}{2}$, which is defined as the completion of~$C_c^\infty(\RR^d)$ with respect to the norm given by:
  \[
    \norm{f}^2_{\HH}=\norm{f}^2_{W^{k,2}(\RR^d)}={\sum_{\nu=0}^k\sum_{\al\in\ZZ_+^d,|\al|\leq\nu}\frac{\nu!}{\al!}
      \binom{k}{\nu} \int_{\RR^d}\abs{\frac{\partial^\nu
          f(x)}{\partial x^\nu}}^2dx}.\] The Sobolev space~$W^{k,2}(\RR^d)$ is a reproducing kernel Hilbert space with the reproducing kernel~$K$, given by (see \cite[Sec.~1.3.5]{Saitoh})
  \[K(x,y)=\frac{1}{(2\pi)^d}\int_{\RR^d}\frac{\exp(i \langle x-y,\xi\rangle)}{(1+\norm{\xi}^2)^k}d\xi, \qquad x,y\in\RR^d,\]
  where~$\norm{\cdot}$ is the Euclidean norm in~$\RR^d$.

  It satisfies Assumption~\ref{assmpt1} with~$\kappa^2 := (2\pi)^{-d}\int_{\RR^d} (1 + \norm{\xi}^2)^{-k}d\xi <\infty$.  We consider the non-linear operator~$A:\HH\to\HH$ given by:
  \begin{equation*}
    [A(f)](x)=\int_{\RR^d}\vartheta(x,s)(f(s))^2d\mu(s), \qquad x\in
    \RR^d,~f\in\mathcal{D}(A)\subset\HH,
  \end{equation*}
  where~$\vartheta:\RR^d\times \RR^d\to\RR$ is~$k$-times differentiable. It can be checked that~$A(f)\in \HH$, with
  \[
    \norm{A(f)}_{\HH} \leq \kappa^2 \norm{f}^2_{\HH} C_k(\theta),
  \]
  where
  \[
    C_k(\theta) :=
    \left[\sum_{\nu=0}^k\sum_{\al\in\ZZ_+^d,|\al|\leq\nu}\frac{\nu!}{\al!}
      \binom{k}{\nu}
      \int_{\RR^d}\left(\int_{\RR^d}\abs{\frac{\partial^\nu
            \vartheta(x,s)}{\partial
            x^\nu}}d\mu(s)\right)^2dx\right]^{1/2}
  \]
  (assumed to be finite).

  The Fr{\'e}chet derivative of~$A$ at~$f$ is given by
  \begin{equation*}
    [A'(f)g](x)=2\int_{\RR^n}\vartheta(x,s)f(s)g(s)d\mu(s).
  \end{equation*}
  Then we have
  \begin{equation*}
    \norm{A'(f_\rho)g}_{\HH}\leq 2\kappa^2\norm{f_\rho}_{\HH}\norm{g}_{\HH} C_k(\theta),
  \end{equation*}
  and
  \begin{equation*}
    \norm{I_K\brac{A'(f)g-A'(f_\rho)g}}_{\mathscr{L}^2(\RR^d,\rho_X;\RR)}\leq \kappa\norm{A'(f)g-A'(f_\rho)g}_{\HH}\leq 2\kappa^3\norm{f-f_\rho}_{\HH}\norm{g}_{\HH} C_k(\theta),
  \end{equation*}
  so that Assumption~\ref{A.assumption} is satisfied.
\end{example}

Under the above non-linearity assumption on the operator~$A$ we now introduce the corresponding operators which will turn out to be useful in the analysis of regularization schemes. 

We recall that~$I_K$ denotes the canonical injection map~$\HH_2 \to \LL$. We define the operator
\begin{align*}\label{Sf}
  B : \HH_1 &\to \LL \\ \nonumber
  f  &\mapsto Bf=B_\rho f := (I_K \circ (A'(f_\rho)))(f)=I_K(A'(f_\rho)f).
\end{align*}

We denote~$T=T_\rho:= B_\rho^{\ast}B_\rho\colon \HH_1\to \HH_1$ the corresponding covariance operator. The operators~$L_K$ from Section~\ref{Definitions}, and ~$T$ are positive, self-adjoint and compact operators, even trace-class operators.

Observe that the operator~$B$ depends on~$I_K$ and~$f_\rho$, thus on the joint probability measure~$\rho$ itself.  It is bounded and satisfies~$\norm{B}_{\HH_1 \to \LL}\leq\kappa L$.

The consistency results as established in Section~\ref{sec:consistency}, yield convergence of the minimizers~$f_{\zz,\la}$, as~$\abs{\zz}=m$ tends to infinity, and the parameter~$\lambda$ is chosen appropriately. However, the rates of convergence may be arbitrarily slow. This phenomenon is known as the no free lunch theorem~\cite{Devroye}. Therefore, we need some prior assumptions on the probability measure~$\rho$ in order to achieve uniform rates of convergence for learning algorithms. 
\begin{assumption}[General source condition]\label{source.cond}
 The true solution~$f_\rho$ belongs to the class~$\Omega(\rho,\phi,R)$ with
    \begin{equation*}
      \Omega(\rho,\phi,R):=\left\{f \in \HH_1: f-\fbar=\phi(T)g \text{ and }\norm{g}_{\HH_1} \leq R\right\},
    \end{equation*}
    where~$\phi$ is a continuous increasing index function defined on the interval~$[0,\kappa^2L^2]$ with the assumption~$\phi(0)=0$.
  \end{assumption}
The general source condition~$f_\rho\in\Omega(\rho,\phi,R)$, by allowing for the index functions~$\phi$, cover a wide range of source conditions, such as H\"{o}lder source condition~$\phi(t)=t^r$ with~$r\geq 0$, and logarithmic-type source condition~$\phi(t)=t^p\log^{-\nu}\left(\frac{1}{t}\right)$ with~$p\in\NN,~\nu\in[0,1]$. The source sets~$\Omega(\rho,\phi,R)$ are precompact sets in~$\HH_1$, since the operator~$T$ is
compact. Observe that in contrast with the linear case, in the equation~$f_\rho-\fbar=\phi(T)g$ from Assumption~\ref{source.cond}, the true solution~$f_\rho$ appears 
on both sides, since the operator~$T$ itself depends on it (through~$A'(f_\rho)$). This condition is more easily interpreted as a condition on the ``initial guess''~$\fbar$, so that the initial error~$(\fbar-f_\rho)$ should satisfy a source condition with respect to the operator linearized at the true solution. Assumption~\ref{source.cond} is usually referred to as a general source condition, see e.g.~\cite{Mathe}, which is a measure of regularity of the true solution~$f_\rho$. This is inspired, on the one hand, by the approach considered in previous works on statistical learning using kernels, and, on the other hand, by the ``classical" literature on non-linear inverse problems. The true solution~$f_\rho$ is represented in terms of the marginal probability distribution~$\rho_X$ over the input space~$X$, and of the linearized operator at the true solution, respectively. Both aspects enter into Assumption~\ref{source.cond}.

Following the concept of Bauer et al.~\cite{Bauer}, and Blanchard et al.~\cite{Blanchard}, we consider the class of probability measures~$\PP_\phi$ which satisfy both the noise assumption~\ref{noise.cond} and which allow for the smoothness assumption~\ref{source.cond}. This class depends on the observation noise distribution  (reflected in the parameters~$M > 0$,~$\Sigma > 0$) and the smoothness properties of the true solution~$f_\rho$ (reflected in the parameters~$R > 0$,~$\phi > 0$).  For the convergence analysis, the output space need not be bounded as long as the noise condition for the output variable is fulfilled. 

The class~$P_{\phi}$ may further be constrained, by imposing properties of the covariance operator~$L_K$ from above. Thus we consider the set of probability  measures~$\PP_{\phi,b}\subset P_{\phi}$ which also satisfy the following condition:
\begin{assumption}[Eigenvalue decay condition]\label{poly.decay}
The eigenvalues~$(t_n)_{n\in\NN}$ of the covariance  operator~$L_K$ follow a polynomial decay, i.e.,\ for fixed positive constants~$\beta$ and~$b>1$,
\begin{equation*}
 t_n\leq\beta n^{-b}~~\forall n\in\NN.
\end{equation*}
\end{assumption}

Now under Assumption~\ref{A.assumption} (ii) using the relation for singular values~$s_{j}(UV) \leq \norm{U} s_{j}(V)$ for~$j\in\NN$ (see Chapter 11 \cite{Pietsch}) we obtain,
\begin{equation*}
s_{j}(T) \leq \norm{A'(f_\rho)}_{\HH_1\to\HH_2}^2 s_{j}(L_K) \leq L^2 s_{j}(L_K).
\end{equation*}
Hence the polynomial decay condition on eigenvalues of the operator $L_K$ implies that the eigenvalues of $T$ also follows the polynomial decay.

We achieve optimal minimax rates of convergence using the concept of \emph{effective dimension} of the operator~$L_K$. For the trace class operator~$L_K$, the effective dimension is defined as
$$
\mathcal{N}(\la)=\mathcal{N}_{L_K}(\la):=\tr\left((L_K+\la
  I)^{-1}L_K\right),\quad  \text{  for }\la>0.
$$
For the infinite dimensional operator~$L_K$, the effective dimension is a continuously decreasing function of~$\la$ from~$\infty$ to~$0$. For further discussion on the effective dimension, we refer to the literature \cite{Lin,Lu16}.

Under Assumptions~\ref{assmpt1},~\ref{A.assumption} (ii), the effective dimension~$\mathcal{N}(\la)$ can trivially be estimated as follows,
\begin{equation}\label{N(l).bd}
  \mathcal{N}(\la)\leq \norm{(L_K+\la
    I)^{-1}}_{\mathcal{L}(\HH_2)}\tr\left(L_K\right) \leq
  \frac{\kappa^2}{\la},\quad \lambda>0.
\end{equation}
However, we know from~\cite[Prop.~3]{Caponnetto} that, under Assumption~\ref{poly.decay}, we have the improved bound
\begin{equation}\label{N(l).bound}
  \mathcal{N}(\la) \leq C_{\beta,b}\la^{-1/b},\quad \text{ for }b>1.
\end{equation}

\subsection{Upper rates of convergence}
In Theorems~\ref{converge}--\ref{convergence}, we present the upper error bounds for the regularized least-squares solution~$f_{\zz,\la}$ over the class of probability measures~$\mathcal{P}_\phi$. We establish the error bounds for both the direct learning setting in the sense of the~$\LL$-norm reconstruction error~$\norm{I_K\brac{A(f_{\zz,\la})-A(f_\rho)}}_{\LL}$ and the inverse problem setting in the sense of the~$\HH_1$-norm reconstruction error~$\norm{f_{\zz,\la}-f_\rho}_{\HH_1}$. Since the explicit expression of~$f_{\zz,\la}$ is not known, we use the definition~\eqref{fzl} of the regularized least-squares solution~$f_{\zz,\la}$ to derive the error bounds. We use the linearization techniques for the operator~$A$ in the neighborhood of the true solution~$f_\rho$ under the (Fr{\'e}chet) differentiability of~$A$. We estimate the error bounds for the regularized least-squares estimator by measuring the complexity of the true solution~$f_\rho$ and the effect of random sampling. The rates of convergence are governed by the noise condition (Assumption~\ref{noise.cond}), the general source condition (Assumption~\ref{source.cond}) and the ill-posedness of the problem, as measured by an assumed power decay~(Assumption~\ref{poly.decay}) of the eigenvalues of~$T$ with exponent~$b > 1$. The effect of random sampling and the complexity of~$f_\rho$ are measured through Assumption~\ref{noise.cond} and Assumption~\ref{source.cond} in Proposition~\ref{main.bound} and Proposition~\ref{approx.err}, respectively. We briefly discuss two additional assumptions of the theorem. Condition~\eqref{l.la.condition.k}  below says that as the regularization parameter~$\la$ decreases, the sample size must increase. This condition will be automatically satisfied under the parameter choice considered later in Theorem~\ref{err.upper.bound.p.para}. The additional assumption~\eqref{smallness} is a ``smallness'' condition which imposes a constraint between~$\norm{w}_{\HH_1}$ and the non-linearity as measured by the parameter~$\gamma$ in Assumption~\ref{A.assumption} (iii). In order for the latter norm to be finite for any function satisfying the source condition~$f_\rho\in \Omega(\rho,\phi,R)$, it requires that~$\phi(t)/\sqrt{t}$ remains bounded near 0, in particular if~$\phi(t)=t^r$, that~$r\geq \frac{1}{2}$.


The error bound discussed in the following theorem holds non-asymptotically, but this holds with sufficiently small regularization parameter~$\la$ and sufficiently large sample size~$m$. For fixed~$\eta$ and~$\la$, we can choose sufficiently large sample size~$m$ such that
\begin{equation}\label{l.la.condition.k}
8\kappa^2\max\paren{1,\frac{L(M+\Sigma)}{\kappa d}}\log\paren{\frac{4}{\eta}} \leq \sqrt{m}\la.
\end{equation}

Under the source condition~$f_\rho-\fbar=\phi(T)g$ for~$\phi(t)=\sqrt{t}\psi(t)$, we have that~$f_\rho-\fbar=T^{1/2}\psi(T)g=T^{1/2}w$ for~$\psi(T)g=w$. We assume that 
\begin{equation}\label{smallness}
2\gamma  \norm{w}_{\HH_1}<1.
\end{equation}

The proofs of Theorems~\ref{converge}--\ref{err.upper.bound.p.para} will be given in Appendix~\ref{sec:proof-upper-rates}.

\begin{theorem}\label{converge}
Let~$\zz$ be i.i.d. samples drawn according to the probability measure~$\rho\in \PP_{\phi}$ where~$\phi(t)=\sqrt{t}$. Suppose Assumptions~\ref{fp}--\ref{assmpt1},~\ref{A.assumption}--\ref{source.cond} and the conditions~\eqref{l.la.condition.k},~\eqref{smallness} hold true. Then, for all~$0<\eta<1$, for the regularized least-squares estimator~$f_{\zz,\la}$ (not necessarily unique) in~(\ref{fzl}) with the confidence~$1-\eta$ the following upper bound holds:
\begin{equation*}
\|f_{\zz,\la}-f_\rho\|_{\HH_1} \leq C_1 \brac{R\sqrt{\la}+\frac{\kappa M}{m\la}+\sqrt{\frac{\Sigma^2\mathcal{N}(\la)}{m\la}}}\log\left(\frac{4}{\eta}\right)
\end{equation*} 
and 
\begin{equation*}
\|I_K\brac{A(f_{\zz,\la})-A(f_\rho})\|_{\LL} \leq C_2\sqrt{\la}\brac{R\sqrt{\la}+\frac{\kappa M}{m\la}+\sqrt{\frac{\Sigma^2\mathcal{N}(\la)}{m\la}}}\log\left(\frac{4}{\eta}\right),
\end{equation*}
where~$C_1$ and~$C_2$ depends on the parameters~$\gamma$,~$L$,~$R$.
\end{theorem}

In the above theorem we discussed the error bounds for the H{\"o}lder source condition (Assumption~\ref{source.cond}) with~$\phi(t)=\sqrt{t}$. In the following theorem, we discuss the error bound for the general source condition with the suitable assumptions on the function~$\phi$. 

\begin{theorem}\label{convergence}
Let~$\zz$ be i.i.d. samples drawn according to the probability measure~$\rho\in \PP_{\phi}$ where~$\phi(t)=\sqrt{t}\psi(t)$ is an index function satisfying the conditions that~$\psi(t)$ and~$\sqrt{t}/\psi(t)$ are nondecreasing functions. Suppose Assumptions~\ref{fp}--\ref{assmpt1}, \ref{A.assumption}--\ref{source.cond} and the conditions~\eqref{l.la.condition.k},~\eqref{smallness} hold true. Then, for all~$0<\eta<1$, for the regularized least-squares estimator~$f_{\zz,\la}$ (not necessarily unique) in~(\ref{fzl}) with the confidence~$1-\eta$ the following upper bound holds:
$$\|f_{\zz,\la}-f_\rho\|_{\HH_1} \leq C \left\{R\phi(\la)+\frac{\kappa M}{m\la}+\sqrt{\frac{\Sigma^2\mathcal{N}(\la)}{m\la}}\right\}\log\left(\frac{4}{\eta}\right),$$
where~$C$ depends on the parameters~$\gamma$,~$L$,~$\norm{w}_{\HH_1}$.
\end{theorem}

Note that error bounds for~$\norm{f_{\zz,\la}-f_\rho}_{\HH_1}$ in both Theorem \ref{converge} and Theorem \ref{convergence} are the same upto a constant factor which depends on the parameters~$\gamma$,~$L$,~$\norm{w}_{\HH_1}$.

In Theorems~\ref{converge}--\ref{convergence}, the error estimates reveal the interesting fact that the error terms consist of increasing and decreasing functions of~$\la$ which led to propose a choice of regularization parameter by balancing the error terms. We derive the rates of convergence for the regularized least-squares estimator based on a data independent (a priori) parameter choice of~$\la$ for the classes of probability measures~$\mathcal{P}_{\phi}$ and~$\mathcal{P}_{\phi,b}$. The effective dimension plays a crucial role in the error analysis of regularized least-squares learning algorithm. In Theorem~\ref{err.upper.bound.p.para}, we derive the rate of convergence for the regularized least-squares solution~$f_{\zz,\la}$ under the general source condition~$f_\rho \in \Omega(\rho,\phi,R)$ for the parameter choice rule for~$\la$ based on the index function~$\phi$ and the sample size~$m$. For the class of probability measures~$\mathcal{P}_{\phi,b}$, the polynomial decay condition (Assumption~\ref{poly.decay}) on the spectrum of the operator~$T$ also enters into the picture and the parameter~$b$ enters in the parameter choice by the estimate~(\ref{N(l).bound}) of effective dimension. For this class, we derive the minimax optimal rate of convergence in terms of the index function~$\phi$, the sample size~$m$ and the parameter~$b$.  

\begin{theorem}\label{err.upper.bound.p.para}
Under the same assumptions of Theorem~\ref{convergence}, the convergence of the regularized least-squares estimator~$f_{\zz,\la}$ in~(\ref{fzl}) to the true solution~$f_\rho$ can be described as:
\begin{enumerate}[(i)]
\item For the class of probability measures~$\mathcal{P}_\phi$ with the parameter choice~$\la=\Theta^{-1}\paren{m^{-1/2}}$ where~$\Theta(t)=t\phi(t)$, we have
$$\mathbb{P}_{\zz\in Z^m}\left\{\|f_{\zz,\la}-f_\rho\|_{\HH_1}\leq C' \phi\left(\Theta^{-1}\paren{m^{-1/2}}\right)\log\left(\frac{4}{\eta}\right)\right\}\geq 1-\eta,$$
where~$C'$ depends on the parameters~$\gamma$,~$L$,~$\norm{w}_{\HH_1}$,~$R$,~$\kappa$,~$M$,~$\Sigma$ and
$$\lim\limits_{\tau\rightarrow\infty}\limsup\limits_{m\rightarrow\infty}\sup\limits_{\rho\in \PP_{\phi}} \mathbb{P}_{\zz\in Z^m}\left\{\|f_{\zz,\la}-f_\rho\|_{\HH_1}>\tau \phi\left(\Theta^{-1}\paren{m^{-1/2}}\right)\right\}=0.$$
\item  For the class of probability measures~$\mathcal{P}_{\phi,b}$ under Assumption~\ref{poly.decay} and the parameter choice~$\la=\Psi^{-1}\paren{m^{-1/2}}$ where~$\Psi(t)=t^{\frac{1}{2}+\frac{1}{2b}}\phi(t)$, we have
$$\mathbb{P}_{\zz\in Z^m}\left\{\|f_{\zz,\la}-f_\rho\|_{\HH_1}\leq C'' \phi\left(\Psi^{-1}\paren{m^{-1/2}}\right)\log\left(\frac{4}{\eta}\right)\right\}\geq 1-\eta,$$
where~$C''$ depends on the parameters~$\gamma$,~$L$,~$\norm{w}_{\HH_1}$,~$R$,~$\kappa$,~$M$,~$\Sigma$,~$b$,~$\beta$ and
$$\lim\limits_{\tau\rightarrow\infty}\limsup\limits_{m\rightarrow\infty}\sup\limits_{\rho\in \PP_{\phi,b}} \mathbb{P}_{\zz\in Z^m}\left\{\|f_{\zz,\la}-f_\rho\|_{\HH_1}>\tau \phi\left(\Psi^{-1}\paren{m^{-1/2}}\right)\right\}=0.$$
\end{enumerate}
\end{theorem}

Notice that the rates given for the class~$P_\phi$ is worse than the one  for the (smaller) class~$P_{\phi,b}$, which is easily seen from the fact  that~$t^{1/2 + 1/(2b)}\geq t$ for~$b>1$, and hence~$\Psi(t) \geq \Theta(t)$ for~$t\in [0,1]$.

We obtain the following corollary as a consequence of Theorem~\ref{err.upper.bound.p.para}.
\begin{corollary}\label{err.upper.bound.cor}
Under the same assumptions of Theorem~\ref{convergence} with the H\"{o}lder's source condition~$f_\rho\in\Omega(\rho,\phi,R),~\phi(t)=t^r$, the convergence of the regularized least-squares estimator~$f_{\zz,\la}$ in~(\ref{fzl}) to the true solution~$f_\rho$ can be described as:
\begin{enumerate}[(i)]
\item For the class of probability measures~$\PP_{\phi}$ with the parameter choice~$\la=m^{-\frac{1}{2r+2}}$, for all~$0<\eta<1$, we have with the confidence~$1-\eta$,
$$\|f_{\zz,\la}-f_\rho\|_{\HH_1}\leq C'm^{-\frac{r}{2r+2}}\log\left(\frac{4}{\eta}\right) \quad\text{for }\frac{1}{2}\leq r\leq 1.$$
\item For the class of probability measures~$\PP_{\phi,b}$ with the parameter choice~$\la=m^{-\frac{b}{2br+b+1}}$, for all~$0<\eta<1$, we have with the confidence~$1-\eta$,
$$\|f_{\zz,\la}-f_\rho\|_{\HH_1}\leq C''m^{-\frac{br}{2br+b+1}}\log\left(\frac{4}{\eta}\right)\quad\text{for }\frac{1}{2}\leq r\leq 1.$$
\end{enumerate}
\end{corollary}

We obtain the following corollary as a consequence of Theorem~\ref{converge}.

\begin{corollary}\label{err.upper.bound.corr}
Under the same assumptions of Theorem~\ref{converge} with the H\"{o}lder's source condition~$f_\rho\in\Omega(\rho,\phi,R),~\phi(t)=t^{1/2}$, the convergence of the regularized least-squares estimator~$f_{\zz,\la}$ in~(\ref{fzl}) to the true solution~$f_\rho$ can be described as:
\begin{enumerate}[(i)]
\item For the class of probability measures~$\PP_{\phi}$ with the parameter choice~$\la=m^{-\frac{1}{3}}$, for all~$0<\eta<1$, we have with the confidence~$1-\eta$,
$$\|f_{\zz,\la}-f_\rho\|_{\HH_1}\leq C_1'm^{-\frac{1}{6}}\log\left(\frac{4}{\eta}\right)$$
and
$$\|I_K\brac{A(f_{\zz,\la})-A(f_\rho})\|_{\LL}\leq C_2'm^{-\frac{1}{3}}\log\left(\frac{4}{\eta}\right),$$
where~$C_1'$ and~$C_2'$ depends on the parameters~$\gamma$,~$L$,~$\norm{w}_{\HH_1}$,~$\kappa$,~$M$,~$\Sigma$.

\item For the class of probability measures~$\PP_{\phi,b}$ with the parameter choice~$\la=m^{-\frac{b}{2b+1}}$, for all~$0<\eta<1$, we have with the confidence~$1-\eta$,
$$\|f_{\zz,\la}-f_\rho\|_{\HH_1}\leq C_1''m^{-\frac{b}{4b+2}}\log\left(\frac{4}{\eta}\right)$$
and
$$\|I_K\brac{A(f_{\zz,\la})-A(f_\rho})\|_{\LL}\leq C_2''m^{-\frac{b}{2b+1}}\log\left(\frac{4}{\eta}\right),$$
where~$C_1''$ and~$C_2''$ depends on the parameters~$\gamma$,~$L$,~$\norm{w}_{\HH_1}$,~$\kappa$,~$M$,~$\Sigma$,~$b$,~$\beta$.
\end{enumerate}
\end{corollary}

Now we compare the error bounds established for the direct learning setting in the sense of~$\LL$-norm reconstruction error~$\norm{I_K\brac{A(f_{\zz,\la})-A(f_\rho)}}_{\LL}$ and the inverse problem setting in the sense of the~$\HH_1$-norm reconstruction error~$\norm{f_{\zz,\la}-f_\rho}_{\HH_1}$. Since under the condition \eqref{l.la.condition.k} from \eqref{c2} we have that~$f_{\zz,\la}\in\mathcal{B}_d(f_\rho)\cap \mathcal{D}(A)\subset\HH_1$ with confidence~$1-\eta/2$, therefore with Assumption \ref{A.assumption} linearizing the operator~$A$ at~$f_\rho$ (i.e.,~$A(f_{\zz,\la})=A(f_\rho)+A'(f_\rho)(f_{\zz,\la}-f_\rho)+r(f_{\zz,\la})$) we conclude that
\begin{align}\label{reconst.err}
\norm{I_K\brac{A(f_{\zz,\la})-A(f_\rho)}}_{\LL} =& \norm{I_K\brac{A'(f_\rho)(f_{\zz,\la}-f_\rho)+r(f_{\zz,\la})}}_{\LL}\\\nonumber
\leq & \norm{B(f_{\zz,\la}-f_\rho)}_{\LL}+\norm{I_Kr(f_{\zz,\la})}_{\LL}\\   \nonumber
\leq & \norm{T^{1/2} (f_{\zz,\la}-f_\rho)}_{\HH_1} +\frac{\gamma}{2}\norm{f_{\zz,\la}-f_\rho}_{\HH_1}^2.
\end{align}

Thus bounding the prediction norm~$\norm{I_K \brac{ A(f_{\zz,\la}) - A(f_\rho)}}_{\LL}$ corresponds to a learning bound in which first norm  consists of some target function~$T^{1/2}f_\rho$ which has additional smoothness~$1/2$, on the other hand the second term is square of the reconstruction error in~${\HH_1}$-norm, therefore this might result in a higher rate. Indeed, this heuristics is validated from the Theorem \ref{converge} and Corollary \ref{err.upper.bound.corr}, where we observe that the prediction norm has the faster convergence rate than the reconstruction error in~${\HH_1}$-norm. 

The assumptions on the non-linear operator~$A$ (see~Assumption~\ref{A.assumption}, and the condition~\eqref{smallness}) allow us to estimate the reconstruction error bounds in~$\HH_1$-norm for H\"older source condition ($\phi(t)=t^r$) corresponds to the range~$\frac{1}{2}\leq r$. It is well-known that Tikhonov regularization has the saturation effect at~$r=1$ (since it has qualification~$1$), therefore we cannot improve the rates of convergence beyond~$r=1$. From \eqref{reconst.err} we observe that for the prediction error we have additional smoothness~$1/2$ in the bound on the right hand side, therefore we only estimate the prediction error for~$r=\frac{1}{2}$.

\subsection{Lower rates of convergence}
In this section, we discuss the lower rates of convergence for non-linear statistical inverse learning problem over a subclass of the probability measures~$\mathcal{P}_{\phi,b}$. The Kullback-Leibler information and Fano inequalities are the main ingredients in the analysis of the estimates for the minimum possible error. Kullback-Leibler divergence between two probability measures~$\rho_1$ and~$\rho_2$ is defined as
$$\mathcal{K}(\rho_1,\rho_2):=\int_Z\log(g(z))d\rho_1(z),$$
where~$g$ is the density of~$\rho_1$ with respect to~$\rho_2$, that is,~$\rho_1(E)=\int_Eg(z)d\rho_2(z)$ for all measurable sets~$E$.

To obtain the lower bound, we define a family of probability measures~$\rho_f$ parameterized by suitable vectors~$f\in\mathcal{D}(A)\subset\HH_1$. We assume that~$Y$ is finite-dimensional space with a basis~$\{v_j\}_{j=1}^d$. Then for each~$f \in\mathcal{D}(A)\subset\HH_1$, we associate the probability measure on the sample space~$Z$:
\begin{equation}\label{p.f}
\rho_f(x,y):=\frac{1}{2dJ}\sum\limits_{j=1}^d\left(a_j(x)\delta_{y+dJv_j}+b_j(x)\delta_{y-dJv_j}\right)\nu(x),
\end{equation}
where~$a_j(x)=J-\langle A(f),K_xv_j\rangle_{\HH_2}$,~$b_j(x)=J+\langle A(f),K_xv_j\rangle_{\HH_2}$,~$J=4\kappa\|A(f)\|_{\HH_2}$ and~$\delta_{y-\xi}$ denotes the Dirac measure on~$Y$ with unit mass at~$y=\xi$. 

Following the analysis of Caponnetto et al. \cite{Caponnetto} and DeVore et al. \cite{DeVore} we establish the lower rates of convergence for the non-linear statistical inverse problem that can be attained by any learning algorithm. The main steps are the following. In order to obtain the lower rates of convergence for learning algorithms, we generate~$N_\varepsilon$-vectors ($f_1,\ldots, f_{N_\varepsilon}$) depending on~$\varepsilon < \varepsilon_0$ for some~$\varepsilon_0 > 0,$ with~$N_\varepsilon \to\infty$ as~$\varepsilon \to 0$ such that any two of these vectors are separated by constant times~$\varepsilon$ with respect to the norm in Hilbert space~$\HH_1$ (Proposition~\ref{fi.fj.kull} (i)). Then we construct the probability measures~$\rho_i=\rho_{f_i}$ from~(\ref{p.f}), parameterized by~$f_i$'s~$(1\leq i\leq N_\varepsilon)$ with small Kullback–Leibler divergence to each other  (Proposition~\ref{fi.fj.kull} (ii)) and are therefore statistically close. Finally, we obtain the lower rates of convergence on applying \cite[Lemma~3.3]{DeVore} using Kullback-Leibler information.


\begin{assumption}\label{A.assumption1}
For the lower rates of convergence, we assume the following conditions on the non-linear operator~$A$:
\begin{enumerate}[(i)]
  \item~$A$ is Fr\'{e}chet differentiable.
  \item The Fr\'{e}chet derivative of~$A$ at the initial guess~$\fbar$ (of the solution of the functional~\eqref{fzl}) is bounded, i.e., there exists~$L < \infty$ such that:
 ~$$\|A'(\fbar)\|_{\HH_1\to\HH_2}\leq L.$$
 
  \item There exists~$\gamma\geq 0$ such that for all~$f,\tilde{f}\in \mathcal{D}(A)\subset\HH_1$ in a sufficiently large ball around~$\fbar$ we have,
 ~$$\norm{I_K\brac{A'(\tilde{f})-A'(f)}}_{HS}\leq \gamma\|\tilde{f}-f\|_{\HH_1}.$$
  \item The function~$\phi$ is a continuous increasing function  with~$\phi(0)=0$ and~$\theta(t)=\phi(t^2)$ is Lipschitz continuous with the constant~$L_\theta$. For the operators~$T=A'(f)^*I_K^*I_KA'(f)$ and~$\overline{T}=A'(\fbar)^*I_K^*I_KA'(\fbar)$:
\begin{equation*}\label{add.assumption}
\phi(T)=R_{f}\phi(\overline{T}) \text{  and  } \norm{R_{f}-I}_{\mathcal{L}(\HH_1)} \leq \zeta\norm{f-\fbar}_{\HH_1},
\end{equation*} 
where~$f$ belongs to the sufficiently large ball~$B_{d}(\fbar)$,~$R_{f}: \HH_1 \to \HH_1$ is a family of bounded linear operators and~$\zeta$ is a positive constant. 

\item The eigenvalues~$(t_n)_{n\in\NN}$ of the operator~$\overline{T}=A'(\fbar)^*I_K^*I_KA'(\fbar)$ follow the polynomial decay: For fixed positive constants~$\al,\beta$ and~$b>1$,
\begin{equation*}
\al n^{-b}\leq t_n\leq\beta n^{-b}~~\forall n\in\NN.
\end{equation*}
\end{enumerate}
\end{assumption}
 
In contrast to upper rates of convergence for Tikhonov regularization, we require the additional assumption (iv) on~$A$ for the lower rates. This condition is the generalization of the following condition used in \cite{Hanke1995} for the Landweber iteration:
\begin{equation*}\label{add.assumption1}
A'(f)=R_{f}A'(\fbar) \text{  and  } \norm{R_{f}-I}_{\mathcal{L}(\HH_1)} \leq \zeta\norm{f-\fbar}_{\HH_1},~f\in B_{d}(\fbar), 
\end{equation*}
which implies that the Fr{\'e}chet derivative of~$A$ is Lipschitz continuous in~$B_d(\fbar)$. Note that in the linear case~$R_f \equiv I$; therefore, Assumption~\ref{A.assumption1} (iv) may be interpreted as a further restriction on the ``non-linearity” of~$A$. 

The proof of the following theorem will be given in Appendix~\ref{sec:proof-lower-rates}.

\begin{theorem}\label{err.lower.bound.k.para}
Let~$\zz$ be i.i.d. samples drawn according to the probability measure~$\rho\in\mathcal{P}_{\phi,b}$ under the hypothesis~$dim(Y)=d<\infty$. Then under Assumptions~\ref{assmpt1},~\ref{A.assumption1} for~$\Psi(t)=t^{\frac{1}{2}+\frac{1}{2b}}\phi(t)$, the estimator~$f_{\zz}^l$ corresponding to any learning algorithm~$l$ ($\zz\to f_\zz^l\in\HH_1$) converges with the following lower rate:
$$\lim\limits_{\tau\to 0}\liminf\limits_{m\rightarrow\infty}\inf\limits_{l\in\mathcal{A}}\sup\limits_{\rho\in\mathcal{P}_{\phi,b}} \mathbb{P}_{\zz\in Z^m}\left\{\|f_\zz^l-f_\rho\|_{\HH_1}>\tau \phi\left(\Psi^{-1}\paren{m^{-1/2}}\right)\right\}=1,$$
where~$\mathcal{A}$ denotes the set of all learning algorithms~$l: \zz\to f^l_\zz.$
\end{theorem}

We obtain the following corollary as a consequence of Theorem~\ref{err.lower.bound.k.para}.
\begin{corollary}
Under the same assumptions of Theorem~\ref{err.lower.bound.k.para}, for any learning algorithm with H\"{o}lder's source condition~$f_\rho\in\Omega(\rho,\phi,R),~\phi(t)=t^r$, the lower rates of convergence can be described as

$$\lim\limits_{\tau\to 0}\liminf\limits_{m\rightarrow\infty}\inf\limits_{l\in\mathcal{A}}\sup\limits_{\rho\in\mathcal{P}_{\phi,b}} \mathbb{P}_{\zz\in Z^m}\left\{\|f_\zz^l-f_\rho\|_{\HH_1}>\tau m^{-\frac{br}{2br+b+1}}\right\}=1.$$
\end{corollary}

The choice of parameter~$\la(m)$ is said to be optimal, if for this choice of parameter, the upper rate of convergence coincides with the minimax lower rate. For the class of probability measures~$\mathcal{P}_{\phi,b}$ with the parameter choice~$\la=\Psi^{-1}(m^{-1/2})$, Theorem~\ref{err.upper.bound.p.para} shares the upper rate of convergence with the lower rate of convergence in Theorem~\ref{err.lower.bound.k.para}. Therefore the choice of the parameter is optimal.


\section{Discussion}\label{Discussion}
Our analysis guarantees the consistency of Tikhonov regularization algorithm and provides a finite sample bound for non-linear statistical inverse learning problem in vector-valued reproducing kernel ansatz, therefore the results can be applied to the multitask learning problem. We also discussed the asymptotic worst-case analysis for any learning algorithm in this setting, showing optimality in the minimax sense on a suitable class of priors. The rates of convergence presented in Section~\ref{sec:convergence} are asymptotic in nature, i.e., all parameters are fixed as~$m\to\infty$. This provides a mathematical foundation for nonlinear inverse problems in the statistical learning framework. The considered framework generalizes previously proposed settings for different learning schemes: direct, linear inverse learning problem.

\subsubsection*{Impact of effective dimension}\label{sec:impact}

The upper rates were represented in terms of the index function~$\phi$ from Assumption~\ref{source.cond}, and the effective dimension~$\mathcal{N}(\la)$ of the governing operator~$L_K$. This is seen from the basic probabilistic bound, given in Proposition~\ref{main.bound}, and this holds regardless of the fact that~$\la\to \mathcal{N}(\la)$ decays at a polynomial rate. However, the construction for the lower bounds makes use of  this constraint. Also, the Corollaries~\ref{err.upper.bound.cor} and~\ref{err.upper.bound.corr} can be given a handy representation of the upper bounds  under power type decay.

\subsubsection*{Saturation effect}\label{sec:saturation}

In Theorem~\ref{converge} we highlighted the upper rates, both for the errors~$\norm{f_{\zz,\la} - f_\rho}_{\HH_1}$, and~$\|I_K\brac{A(f_{\zz,\la})-A(f_\rho})\|_{\LL}~$ in the limiting case when smoothness is given through the index function~$\phi(t) = \sqrt{t}$;  and these differ by a factor~$\sqrt\la.$ We emphasize that for higher smoothness~$\phi(t) = \sqrt \psi(t)$ with an additional index function~$\psi$ this cannot be expected to remain valid. This is known from the linear case and is due to the saturation effect of Tikhonov regularization.

\subsubsection*{Relation to classical regularization theory}\label{sec:relation}

Within the present study, the smoothness assumption~\ref{source.cond} is based on the composed operator~$B= I_{K} \circ [A^{\prime}(f_\rho)]$ through~$T= B^{\ast}B$. This is in contrast to classical regularization theory, when the corresponding operator is~$\left[ A^{\prime}(f_\rho)\right]^{\ast} A^{\prime}(f_\rho)$. By assuming an appropriate link condition between the operators~$\left[ A^{\prime}(f_\rho)\right]^{\ast} A^{\prime}(f_\rho)$ and~$T$ one can transfer the obtained rates results from the present context to the standard ones, and we refer to the corresponding calculus established in~\cite{pm2018a}.

\subsubsection*{Parameter choice}\label{sec:choice}

The a-priori parameter choice considered in our analysis depends on the smoothness parameters~$b$,~$\phi$. In practice, a posteriori parameter choice rule (data-dependent) for the regularization parameter~$\la$ such as the discrepancy principle, balancing principle, quasi-optimality principle with theoretical justification is required, so that we can turn our results to data-dependent minimax adaptivity even in the absence of a priori knowledge of the regularity parameters.

\appendix

\section{Notation and probabilistic estimates}

Here we introduce some relevant operators.
\begin{definition}[Sampling operator]
  For a discrete ordered set~$\xx=(x_i)_{i=1}^m$, the sampling
    operator~$S_\xx:\HH_2 \to Y^m$ is defined as
$$S_\xx(f):=(f(x_1),\ldots,f(x_m)).$$
\end{definition}

We equip the product Hilbert space~$Y^m$ with the scalar product
$\inner{(y_i)_{i=1}^m,(y'_i)_{i=1}^m} = \frac{1}{m}\sum_{i=1}^m
\inner{y_i,y'_i},$ and denote the associated Hilbert
norm~$\norm{\yy}^2_m=\frac{1}{m}\sum_{i=1}^m\norm{y_i}_Y^2$
for~$\yy=(y_1,\ldots,y_m)$. Then the adjoint~$S_\xx^*:Y^m\to\HH_2$ is
given by
$$S_\xx^*\mathbf{c}=\frac{1}{m}\sum_{i=1}^m K_{x_i} c_i,~~~~\forall \mathbf{c}=(c_1,\ldots,c_m)\in Y^m.$$

Under Assumption~\ref{assmpt1}, the sampling operator is bounded by~$\kappa$, since
 \[ \norm{S_\xx f}^2_m 
   =\frac{1}{m}\sum_{i=1}^m\norm{f(x_i)}_Y^2
   =\frac{1}{m}\sum_{i=1}^m\norm{K_{x_i}^*f}_Y^2 \leq
   \kappa^2\norm{f}^2_{\HH_1}.\]

The sampling versions are the operators~$B_\xx:=S_\xx\circ (A'(f_\rho))$ and~$T_\xx:=B_\xx^*B_\xx$. The operator~$T_\xx$ is positive and self-adjoint. Under the Assumptions~\ref{assmpt1},~\ref{A.assumption} (ii), the operator~$B_\xx$ is bounded and satisfies~$\norm{B_\xx}_{\HH_1 \to Y^m}\leq\kappa L$. We also recall that~$L_K=I_K^*I_K$ for the canonical injection map~$I_K:\HH_2\to\LL$ and~$T=B^*B$ for~$B=I_K\circ [A'(f_\rho)]$. These operators will be used in our analysis.

The following inequality is based on the results of Pinelis and Sakhanenko \cite{Pinelis}.

\begin{proposition}\label{pinels_lemma}
Let~$\HH$ be a real separable Hilbert space and~$\xi$ be a random variable on~$(\Omega,\rho)$ with values in~$\HH$. If there exist two constant~$Q$ and~$S$ satisfying
$$\mathbb{E}_\omega\left\{\|\xi(\omega)-\mathbb{E}_\omega(\xi)\|_{\HH}^n\right\} \leq \frac{1}{2}n!S^2Q^{n-2}~~~\forall n \geq 2,$$
then for any~$0<\eta<1$ and for all~$m \in \NN$,
$$\mathbb{P}\left\{(\omega_1,\ldots,\omega_m) \in \Omega^m : \left\|\frac{1}{m}\sum\limits_{i=1}^{m}\xi(\omega_i)-\mathbb{E}_\omega(\xi)\right\|_{\HH} \leq 2\left(\frac{Q}{m}+\frac{S}{\sqrt{m}}\right)\log\left(\frac{2}{\eta}\right)\right\} \geq 1-\eta.$$
\end{proposition}


In the following proposition, we measure the effect of random sampling using Assumption~\ref{noise.cond}. The quantities describe the probabilistic estimates of the perturbation measure due to random sampling. These bounds are standard in learning theory.

\begin{proposition}\label{main.bound}
Let~$\zz$ be i.i.d. random samples with Assumptions~\ref{fp}--\ref{assmpt1}, then for~$m \in \NN$ and~$0<\eta<1$, each of the following estimate holds with the confidence~$1-\eta$,
\begin{equation*}
\norm{\frac{1}{m}\sum\limits_{i=1}^m(L_K+\la I)^{-1/2}K_{x_i}(y_i-A(f_\rho)(x_i))}_{\HH_2} \leq 2\left(\frac{\kappa M}{m\sqrt{\la}}+\sqrt{\frac{\Sigma^2\mathcal{N}(\la)}{m}}\right)\log\left(\frac{2}{\eta}\right),
\end{equation*}

\begin{equation*}
\norm{\frac{1}{m}\sum\limits_{i=1}^m K_{x_i}(y_i-A(f_\rho)(x_i))}_{\HH_2} \leq 2\left(\frac{\kappa M}{m}+\frac{\kappa\Sigma}{\sqrt{m}}\right)\log\left(\frac{2}{\eta}\right)
\end{equation*}
%
and
\begin{equation*}
\|S_\xx^*S_\xx-L_K\|_{\mathcal{L}_2(\HH_2)}\leq 2\left(\frac{\kappa^2}{m}+\frac{\kappa^2}{\sqrt{m}}\right)\log\left(\frac{2}{\eta}\right).
\end{equation*}
\end{proposition}
\begin{proof}
To estimate the first expression, we consider the random variable~$\xi_1(z)= (L_K+\la I)^{-1/2}K_x(y-A(f_\rho)(x))$ from~$(Z,\rho)$ to reproducing kernel Hilbert space~$\HH_2$. Under the Assumption \ref{fp} we obtain
$$\mathbb{E}_z(\xi_1)=\int_Z(L_K+\la I)^{-1/2}K_x(y-A(f_\rho)(x))d\rho(x,y)=0,$$
$$\frac{1}{m}\sum\limits_{i=1}^m\xi_1(z_i)=\frac{1}{m}\sum\limits_{i=1}^m(L_K+\la I)^{-1/2}K_{x_i}(y_i-A(f_\rho)(x_i))$$
and
\begin{eqnarray*}
\mathbb{E}_z(\|\xi_1-\mathbb{E}_z(\xi_1)\|_{\HH_2}^n)&=&\mathbb{E}_z\left(\|(L_K+\la I)^{-1/2}K_x(y-A(f_\rho)(x))\|_{\HH_2}^n\right) \\
&\leq& \mathbb{E}_z\left(\|K_x^*(L_K+\la I)^{-1}K_x\|_{\mathcal{L}(Y)}^{n/2}\|y-A(f_\rho)(x)\|_{Y}^n\right).
\end{eqnarray*}
Under Assumptions~\ref{noise.cond}--\ref{assmpt1} we get~$\int_Y\norm{y-A(f_\rho)(x)}_Y^nd\rho(y|x)\leq \frac{n!}{2}\Sigma^2M^{n-2},\quad\forall n\geq 2$ which implies,
$$\mathbb{E}_z\left(\|\xi_1-\mathbb{E}_z(\xi_1)\|_{\HH_2}^n\right) \leq \frac{n!}{2}(\Sigma\sqrt{\mathcal{N}(\la)})^2\left(\frac{\kappa M}{\sqrt{\la}}\right)^{n-2},\quad\forall n\geq 2.$$
On applying Proposition~\ref{pinels_lemma} with~$Q=\kappa M$ and~$S=\Sigma\sqrt{\mathcal{N}(\la)}$ follows that,
$$\mathbb{P}_{\zz \in Z^m}\left\{\norm{\frac{1}{m}\sum\limits_{i=1}^m(L_K+\la I)^{-1/2}K_{x_i}(y_i-A(f_\rho)(x_i))}_{\HH_2} \leq 2\left(\frac{\kappa M}{m\sqrt{\la}}+\sqrt{\frac{\Sigma^2\mathcal{N}(\la)}{m}}\right)\log\left(\frac{2}{\eta}\right)\right\} \geq 1-\eta.$$ 

The second expression can be estimated easily by considering the random variable~$\xi_2(z)= K_x(y-A(f_\rho)(x))$ from~$(Z,\rho)$ to reproducing kernel Hilbert space~$\HH_2$, while the proof of the third expression can be obtained from Theorem~2 in De Vito et al. \cite{DeVito0}.
\end{proof}

\begin{proposition}\label{I1}
For~$m \in \NN$ and~$0<\eta<1$, under with Assumptions~\ref{assmpt1}, the following estimates hold with the confidence~$1-\eta/2$,
\begin{equation*}
\|S_\xx^*S_\xx-L_K\|_{\mathcal{L}_2(\HH_2)}\leq \frac{\la}{2},
\end{equation*}

\begin{equation*}
\|(S_\xx^*S_\xx+\la I)^{-1}(L_K+\la I)\|_{\mathcal{L}_2(\HH_2)}\leq 2
\end{equation*}
and
\begin{equation*}
\|(S_\xx^*S_\xx+\la I)^{-1/2}(L_K+\la I)^{1/2}\|_{\mathcal{L}_2(\HH_2)}\leq \sqrt{2}
\end{equation*}
provided that
\begin{equation}\label{l.la.condition}
8\kappa^2\log(4/\eta) \leq \sqrt{m}\la.
\end{equation}
\end{proposition}
\begin{proof}
From Proposition~\ref{main.bound} under the Assumptions~\ref{assmpt1}, the following inequality holds with the confidence~$1-\eta/2$,
\begin{equation*}
\|S_\xx^*S_\xx-L_K\|_{\mathcal{L}_2(\HH_2)}\leq 2\left(\frac{\kappa^2}{m}+\frac{\kappa^2}{\sqrt{m}}\right)\log\left(\frac{4}{\eta}\right).
\end{equation*}

Then under the condition~(\ref{l.la.condition}), we get with the confidence~$1-\eta/2$,
\begin{equation*}
\|S_\xx^*S_\xx-L_K\|_{\mathcal{L}_2(\HH_2)}\leq \frac{4\kappa^2}{\sqrt{m}}\log\left(\frac{4}{\eta}\right)\leq \frac{\la}{2}
\end{equation*} 

which implies
\begin{equation}\label{LK.I.LK.Sx}
\|(L_K+\la I)^{-1}(S_\xx^*S_\xx-L_K)\|_{\mathcal{L}_2(\HH_2)}\leq \frac{1}{2}.
\end{equation}

For the second inequality, we consider
$$(S_\xx^*S_\xx+\la I)^{-1}(L_K+\la I)=\{I-(L_K+\la I)^{-1}(L_K-S_\xx^*S_\xx)\}^{-1}$$
which implies
\begin{equation}\label{SS.LK}
\|(S_\xx^*S_\xx+\la I)^{-1}(L_K+\la I)\|_{\mathcal{L}_2(\HH_2)}\leq \sum\limits_{n=0}^\infty\|(L_K+\la I)^{-1}(L_K-S_\xx^*S_\xx)\|_{\mathcal{L}_2(\HH_2)}^n.
\end{equation}


Consequently, using~(\ref{LK.I.LK.Sx}) in the inequality~(\ref{SS.LK}) we obtain with the probability~$1-\eta/2$,
\begin{equation*}
\|(S_\xx^*S_\xx+\la I)^{-1}(L_K+\la I)\|_{\mathcal{L}_2(\HH_2)}\leq 2.
\end{equation*}

Applying \cite[Prop.~5.7]{Blanchard} we get with the probability~$1-\eta/2$,
\begin{equation*}
\|(S_\xx^*S_\xx+\la I)^{-1/2}(L_K+\la I)^{1/2}\|_{\mathcal{L}_2(\HH_2)}\leq\|(S_\xx^*S_\xx+\la I)^{-1}(L_K+\la I)\|_{\mathcal{L}_2(\HH_2)}^{1/2}\leq \sqrt{2}.
\end{equation*}
\end{proof}

\section{Proof of the consistency results}
\label{sec:proof-consistency}

Throughout the analysis we use the following identity in the real Hilbert space~$\HH$:
\begin{equation*}
\norm{f-h}_{\HH}^2-\norm{f-g}_{\HH}^2 =\norm{g-h}_{\HH}^2-2\inner{f-g,h-g}_{\HH}\qquad f,g,h\in\HH.
\end{equation*}

\begin{proof}[Proof of Theorem \ref{thm:consistency-rms}]
  By the definition of~$f_{\zz,\la}$ as a solution to the minimization problem~(\ref{fzl}), we get the inequality
\begin{equation*}
  \norm{S_\xx A(f_{\zz,\la})-\yy}^2_m
  +\la\|f_{\zz,\la}-\fbar\|_{\HH_1}^2 \leq
  \norm{S_\xx A(f_\rho)-\yy}^2_m
  +\la\|f_\rho-\fbar\|_{\HH_1}^2.
\end{equation*}
It follows that
\begin{multline}
  \label{c1}
  \norm{S_\xx \brac{A(f_{\zz,\la})-A(f_\rho)}}^2_m 
  +\la\|f_{\zz,\la}-\fbar\|_{\HH_1}^2 
  \leq 2\inner{S_\xx \brac{A(f_\rho)-A(f_{\zz,\la})},S_\xx A(f_\rho)-\yy}_m
  +\la\|f_\rho-\fbar\|_{\HH_1}^2.
\end{multline}

Consequently, we get
\begin{align}\label{c11}
&\norm{I_K\brac{A(f_{\zz,\la})-A(f_\rho)}}_{\mathscr{L}^2(X,\rho_X;Y)}^2+\la\|f_{\zz,\la}-\fbar\|_{\HH_1}^2 \\  \nonumber
\leq &  2\inner{A(f_\rho)-A(f_{\zz,\la}),S_\xx^*\{S_\xx A(f_\rho)-\yy\}}_{\HH_2}+\la\|f_\rho-\fbar\|_{\HH_1}^2\\ \nonumber
&+\inner{\paren{L_K-S_\xx^*S_\xx}\brac{A(f_{\zz,\la})-A(f_\rho)},A(f_{\zz,\la})-A(f_\rho)}_{\HH_2}
\end{align}
and
\begin{equation*}
\|f_{\zz,\la}-f_\rho\|_{\HH_1}^2 \leq \frac{2}{\la}\|A(f_{\zz,\la})-A(f_\rho)\|_{\HH_2}\|S_\xx^*\{S_\xx A(f_\rho)-\yy\}\|_{\HH_2}+2\inner{f_\rho-f_{\zz,\la},f_\rho-\fbar}_{\HH_1}.
\end{equation*}

Under the Lipschitz continuity of the operator~$A$ (Assumption~\ref{ass:Lipschitz}) (i.e.,~$\norm{A(f)-A(f_\rho)}_{\HH_2}\leq L\norm{f-f_\rho}_{\HH_1}$ for~$f\in\HH_1$) we get,
\begin{equation*}
\|f_{\zz,\la}-f_\rho\|_{\HH_1}^2 \leq \frac{2L}{\la}\|f_{\zz,\la}-f_\rho\|_{\HH_1}\|S_\xx^*\{S_\xx A(f_\rho)-\yy\}\|_{\HH_2}+2\norm{f_{\zz,\la}-f_\rho}_{\HH_1}\norm{f_\rho-\fbar}_{\HH_1}
\end{equation*}
which implies
\begin{equation}\label{fz.bound}
\|f_{\zz,\la}-f_\rho\|_{\HH_1} \leq \frac{2L}{\la}\|S_\xx^*\{S_\xx A(f_\rho)-\yy\}\|_{\HH_2}+2\norm{f_\rho-\fbar}_{\HH_1}.
\end{equation}

Now squaring both sides and taking expectation with respect to~$\zz$ we obtain,
\begin{equation}\label{fz.bound1}
\mathbb{E}_\zz\paren{\|f_{\zz,\la}-f_\rho\|_{\HH_1}^2} \leq \frac{8L^2}{\la^2}\mathbb{E}_\zz\paren{\|S_\xx^*\{S_\xx A(f_\rho)-\yy\}\|_{\HH_2}^2}+8\norm{f_\rho-\fbar}_{\HH_1}^2.
\end{equation}

Under Assumptions~\ref{fp},~\ref{assmpt1} and~$\sigma_\rho^2=\mathbb{E}_z\paren{\norm{y-A(f_\rho)(x)}_Y^2}<\infty$ we have that
\begin{align}\label{con.ineq1}
\mathbb{E}_\zz\paren{\|S_\xx^*\{S_\xx A(f_\rho)-\yy\}\|_{\HH_2}^2}
&=\frac{1}{m^2}\mathbb{E}_\zz\paren{\sum\limits_{i,j=1}^m
\inner{K_{x_i}\brac{y_i-A(f_\rho)(x_i)},K_{x_j}\brac{y_j-A(f_\rho)(x_j)}}_{\HH_2}}\\   \nonumber
&=\frac{1}{m^2}\mathbb{E}_\zz\paren{\sum\limits_{i=1}^m
\norm{K_{x_i}\brac{y_i-A(f_\rho)(x_i)}}_{\HH_2}^2}\\  \nonumber
&\leq\frac{\kappa^2}{m^2}\mathbb{E}_\zz\paren{\sum\limits_{i=1}^m
\norm{y_i-A(f_\rho)(x_i)}_{Y}^2}=\frac{\kappa^2\sigma_\rho^2}{m},
\end{align}
and from \cite[Lemma~1]{Smale2} we have,
\begin{equation}\label{con.ineq2}
\mathbb{E}_\xx\paren{\|S_\xx^*S_\xx-L_K\|_{\mathcal{L}(\HH_2)}^2}\leq \frac{\kappa^2}{m}.
\end{equation}

Using~\eqref{con.ineq1} in~\eqref{fz.bound1} we get,
\begin{equation}\label{c2.exp}
\mathbb{E}_\zz\paren{\|f_{\zz,\la}-f_\rho\|_{\HH_1}^2} \leq \frac{8\kappa^2\sigma_\rho^2L^2}{\la^2m} +8\|f_\rho-\fbar\|_{\HH_1}^2
\end{equation}
from which with the parameter choice rule~(\ref{la.choice1}) we deduce that
\begin{equation}\label{E.fzl.bound}
\limsup\limits_{m\to\infty}\mathbb{E}_\zz\paren{\|f_{\zz,\la}-f_\rho\|_{\HH_1}^2} \leq 8\|f_\rho-\fbar\|_{\HH_1}^2.
\end{equation}

Hence, we observe that~$a^2:=\sup\limits_{m\in\NN} \mathbb{E}_\zz\paren{\norm{f_{\zz,\la}-f_\rho}_{\HH_1}^2}<\infty$. Now, we show that there exists a subsequence of~$(f_{\zz,\la})_{m\in\NN}$, denoted by~$(f_{\zz(k),\la})_{k\in\NN}$, such that
\begin{equation}\label{Exp.fz}
\mathbb{E}_{\zz(k)}\paren{\inner{f_{\zz(k),\la}-f_\rho,f}_{\HH_1}}\to\inner{\tilde{f},f}_{\HH_1} \text{  as }~k\to\infty
\end{equation} 
for some~$\tilde{f}\in\HH_1$ and for all~$f\in\HH_1$.

Let~$\brac{e_i:i\in\NN}$ be a complete orthonormal basis of the separable Hilbert space~$\HH_1$. By the Cauchy-Schwarz inequality, we have~$\abs{\EE_\zz\paren{\inner{f_{\zz,\la}-f_\rho,e_1}_{\HH_1}}}^2\leq\EE_\zz\paren{\norm{f_{\zz,\la}-f_\rho}_{\HH_1}^2}\norm{e_1}_{\HH_1}^2<\infty$. Hence, there exists a subsequence~$(f_{\zz(k),\la})_{k\in\NN}$ such that~$\EE_{\zz(k)}\paren{\inner{f_{\zz(k),\la}-f_\rho,e_1}_{\HH_1}}\to\xi_1$ as~$k\to\infty$ for some~$\xi_1\in\RR$. Repeating the same arguments, we again get a subsequence~$(f_{\zz(k),\la})_{k\in\NN}$ such that~$\EE_{\zz(k)}\paren{\inner{f_{\zz(k),\la}-f_\rho,e_2}_{\HH_1}}\to\xi_2$ as~$k\to\infty$ for some~$\xi_2\in\RR$ and so on.
Therefore, we obtain a diagonal sequence~$(f_{\zz(k),\la})_{k\in\NN}$ with the property:
$\EE_{\zz(k)}\paren{\inner{f_{\zz(k),\la}-f_\rho,e_i}_{\HH_1}}\to\xi_i$ as~$k\to\infty$ for all~$i\in\NN$. For all~$\ell\in\NN$ we have,
\begin{equation*}
\sum\limits_{i=1}^\ell\xi_i^2\leq \lim\limits_{k\to\infty}\EE_{\zz(k)}\paren{\sum\limits_{i=1}^\ell\abs{\inner{f_{\zz(k),\la}-f_\rho,e_i}_{\HH_1}}^2} \leq \limsup\limits_{k\to\infty}\EE_{\zz(k)}\paren{\norm{f_{\zz(k),\la}-f_\rho}_{\HH_1}^2} \leq a^2.
\end{equation*}

Hence,~$\tilde{f}:=\sum\limits_{i=1}^\infty \xi_i e_i$ is well defined. For~$\varepsilon>0$, choose~$\ell\in\NN$ such that~$\sum\limits_{i=\ell+1}^\infty\inner{f,e_i}_{\HH_1}^2\leq \paren{\frac{\varepsilon}{4a}}^2$. There exists~$K>0$ such that
\begin{equation*}
\abs{\sum\limits_{i=1}^\ell\inner{f,e_i}_{\HH_1}\EE_{\zz(k)}\paren{\inner{\tilde{f}+f_\rho-f_{\zz(k),\la},e_i}_{\HH_1}}}\leq\frac{\varepsilon}{2}
\end{equation*}
for~$k\geq K$. Now, it follows from the Cauchy-Schwarz inequality that~$\abs{\EE_{\zz(k)}\paren{\inner{\tilde{f}+f_\rho-f_{\zz(k),\la},f}_{\HH_1}}}\leq\varepsilon$ for~$k\geq K$. This proves the claim~\eqref{Exp.fz}. 

From inequality~(\ref{c11}) we get:
\begin{align*}
&\mathbb{E}_\zz\paren{\|I_K\{A(f_{\zz,\la})-A(f_\rho)\}\|_{\mathscr{L}^2(X,\rho_X;Y)}^2}+\la\mathbb{E}_\zz\paren{\|f_{\zz,\la}-\fbar\|_{\HH_1}^2} \\ 
\leq  & 2 \mathbb{E}_\zz\paren{\|S_\xx^*\{S_\xx A(f_\rho)-\yy\}\|_{\HH_2}\|A(f_{\zz,\la})-A(f_\rho)\|_{\HH_2}}\\
&+\mathbb{E}_\zz\paren{\|S_\xx^*S_\xx-L_K\|_{\mathcal{L}(\HH_2)}\|A(f_{\zz,\la})-A(f_\rho)\|_{\HH_2}^2}+\la\|f_\rho-\fbar\|_{\HH_1}^2\\
\leq  & 2 \sbrac{\mathbb{E}_\zz\paren{\|S_\xx^*\{S_\xx A(f_\rho)-\yy\}\|_{\HH_2}^2}}^{1/2}\sbrac{\mathbb{E}_\zz\paren{\|A(f_{\zz,\la})-A(f_\rho)\|_{\HH_2}^2}}^{1/2}\\
&+\sbrac{\mathbb{E}_\zz\paren{\|S_\xx^*S_\xx-L_K\|_{\mathcal{L}(\HH_2)}^2}}^{1/2}\sbrac{\mathbb{E}_\zz\paren{\|A(f_{\zz,\la})-A(f_\rho)\|_{\HH_2}^4}}^{1/2}+\la\|f_\rho-\fbar\|_{\HH_1}^2.
\end{align*}

Under the Lipschitz continuity of~$A$, from~\eqref{con.ineq1},~\eqref{con.ineq2},~\eqref{E.fzl.bound} with the parameter choice rule~(\ref{la.choice1}) we obtain
\begin{equation}\label{Exp.Afz}
\mathbb{E}_\zz\paren{\|I_K\{A(f_{\zz,\la})-A(f_\rho)\}\|_{\mathscr{L}^2(X,\rho_X;Y)}^2}\to 0  \text{  as }~m\to\infty.
\end{equation}

We have~$\mathcal{D}(A)$ is weakly closed and~$A:\HH_1\to\HH_2$ is Lipschitz continuous, this implies that~$I_KA:\HH_1\to\LL$ is weakly sequentially closed. 

Now from~\eqref{Exp.fz},~\eqref{Exp.Afz} we obtain a subsequence again denoted by~$(f_{\zz(k),\la})_{k\in\NN}$ such that~$\inner{f_{\zz(k),\la}-f_\rho,f}_{\HH_1}\to\inner{\tilde{f},f}_{\HH_1}$
for some~$\tilde{f}\in\HH_1$, for all~$f\in\HH_1$ and~$\|I_K\{A(f_{\zz(k),\la})-A(f_\rho)\}\|_{\mathscr{L}^2(X,\rho_X;Y)}\to 0$ as~$k\to\infty$ almost surely, hence the weak closedness and one-to-one assumption on assumption on~$I_KA$ imply that~$\tilde{f}=0$.

Our next aim is to prove the convergence~\eqref{consistency.rate.exp}.
By contradiction, assume that there exists an~$\varepsilon> 0$ and a subsequence~$(f_{\zz(k),\la})_{k\in\NN}$ such that
\begin{equation}\label{E.c3}
\mathbb{E}_{\zz(k)}\paren{\|f_{\zz(k),\la}-f_\rho\|_{\HH_1}^2}\geq\varepsilon  \text{     for all  } k \in \NN.
\end{equation}

We have the identity
\begin{equation}\label{E.c4}
\|f_{\zz(k),\la}-f_\rho\|_{\HH_1}^2=\|f_{\zz(k),\la}-\fbar\|_{\HH_1}^2+\|f_\rho-\fbar\|_{\HH_1}^2+2\langle \fbar-f_{\zz(k),\la},f_\rho-\fbar\rangle_{\HH_1}.
\end{equation}

Using the same arguments as above, we can again find a further subsequence~$(f_{\zz(k),\la})_{k\in\NN}$ such that
$\mathbb{E}_\zz\inner{f_{\zz(k),\la}-f_\rho,f}\to 0$
for all~$f\in \HH_1$ as~$k\to \infty$. Hence from the inequalities~(\ref{E.fzl.bound}) and~(\ref{E.c4}) we obtain
\begin{align*}
\limsup\limits_{k\to\infty}\mathbb{E}_{\zz(k)}\paren{\|f_{\zz(k),\la}-f_\rho\|_{\HH_1}^2}
&\leq 2\|f_\rho-\fbar\|_{\HH_1}^2+2\limsup\limits_{k\to\infty}\mathbb{E}_{\zz(k)}\paren{\langle \fbar-f_{\zz(k),\la},f_\rho-\fbar\rangle_{\HH_1}}\\
&=2\limsup\limits_{k\to\infty}\mathbb{E}_{\zz(k)}\paren{\langle f_\rho-f_{\zz(k),\la},f_\rho-\fbar\rangle_{\HH_1}}=0,
\end{align*}
which contradicts~(\ref{E.c3}). This completes the proof of the desired result~(\ref{consistency.rate.exp}).
\end{proof}

\vspace{4mm}

\begin{proof}[Proof of Theorem \ref{thm:asconsistency}]
From the inequality~\eqref{fz.bound} and Proposition~\ref{main.bound} under Assumptions~\ref{assmpt2}--\ref{ass:Lipschitz}, the following inequality holds with the confidence~$1-\eta/2$:
\begin{equation}\label{c2}
\|f_{\zz,\la}-f_\rho\|_{\HH_1} \leq \frac{4\kappa(M+\Sigma)L}{\la\sqrt{m}} \log\left(\frac{4}{\eta}\right) +2\|f_\rho-\fbar\|_{\HH_1}.
\end{equation}

Choosing the parameter~$\eta(m)=4/m^2$, we obtain
\begin{equation*}
\mathbb{P}_{\zz\in Z^m}\brac{E_m: \norm{f_{\zz,\la}-f_\rho}_{\HH_1}> 8\kappa L(M+\Sigma)\frac{\log m}{\la \sqrt{m}}+2\norm{f_\rho-\fbar}_{\HH_1}}\leq \frac{2}{m^2}.
\end{equation*}

Therefore, the sum of the probabilities of the events~$E_m$ is finite:
\begin{equation*}
\sum\limits_{m=1}^\infty \mathbb{P}_{\zz\in Z^m}\paren{E_m}\leq\sum\limits_{m=1}^\infty \frac{2}{m^2}<\infty.
\end{equation*}

Hence applying the Borel-Cantelli lemma we get,
\begin{equation*}
\mathbb{P}_{\zz\in Z^m}\paren{\limsup\limits_{m\to\infty} E_m}=0
\end{equation*}
from which with the parameter choice rule~(\ref{la.choice}) we deduce that
\begin{equation}\label{fzl.bound}
\limsup\limits_{m\to\infty}\|f_{\zz,\la}-f_\rho\|_{\HH_1} \leq 2\|f_\rho-\fbar\|_{\HH_1}
\end{equation}
almost surely. Note that~$f_{\zz,\la}$ is finite almost surely due to~(\ref{fzl.bound}). Hence, there exists a subsequence of~$(f_{\zz,\la})_{m\in\NN}$ which weakly converges to some~$\tilde{f}$. We denote the subsequence by~$(f_{\zz(k),\la})_{k\in\NN}$, i.e.,~$f_{\zz(k),\la}\rightharpoonup \tilde{f}$. The next step of the proof is to show that~$\tilde{f}=f_\rho$.

From inequality~(\ref{c11}) and Proposition~\ref{main.bound} under Assumptions~\ref{assmpt2}--\ref{assmpt1}, the following inequality holds with  confidence~$1-\eta$,
\begin{align*}
  \|I_K\{A(f_{\zz,\la})-A(f_\rho)\}\|_{\mathscr{L}^2(X,\rho_X;Y)}^2
  \leq & \frac{4\kappa(M+\Sigma)}{\sqrt{m}}\|A(f_{\zz,\la})-A(f_\rho)\|_{\HH_2} \log\left(\frac{4}{\eta}\right) \\ 
&  +\frac{4\kappa^2}{\sqrt{m}}\|A(f_{\zz,\la})-A(f_\rho)\|_{\HH_2}^2\log\left(\frac{4}{\eta}\right) +\la\|f_\rho-\fbar\|_{\HH_1}^2.
\end{align*}

Using the arguments similar to above, with the parameter choice rule~(\ref{la.choice}) we obtain~$\|I_K\{A(f_{\zz,\la})-A(f_\rho)\}\|_{\mathscr{L}^2(X,\rho_X;Y)}\to 0$ almost surely. Now~$f_{\zz(k),\la}\rightharpoonup \tilde{f}$ in~$\HH_1$ and~$\|I_K\{A(f_{\zz(k),\la})-A(f_\rho)\}\|_{\mathscr{L}^2(X,\rho_X;Y)}\to 0$ as~$k\to \infty$ almost surely, hence the weak closedness and one-to-one assumption on assumption on~$I_KA$ imply that~$\tilde{f}=f_\rho$.

Our next aim is to prove the convergence~\eqref{consistency.rate}.
By contradiction, assume that there exists an~$\varepsilon> 0$ and a subsequence~$(f_{\zz(k),\la})_{k\in\NN}$ such that
\begin{equation}\label{c3}
\|f_{\zz(k),\la}-f_\rho\|_{\HH_1}^2\geq\varepsilon.
\end{equation}

We have the identity
\begin{equation}\label{c4}
\|f_{\zz(k),\la}-f_\rho\|_{\HH_1}^2=\|f_{\zz(k),\la}-\fbar\|_{\HH_1}^2+\|f_\rho-\fbar\|_{\HH_1}^2+2\langle \fbar-f_{\zz(k),\la},f_\rho-\fbar\rangle_{\HH_1}.
\end{equation}

Using the same arguments as above, we can again find a further subsequence~$(f_{\zz(k),\la})_{k\in\NN}$ which weakly converges to~$f_\rho$. Hence from the inequalities~(\ref{fzl.bound}) and~(\ref{c4}) we obtain almost surely,
\begin{align*}
\limsup\limits_{m\to\infty}\|f_{\zz(k),\la}-f_\rho\|_{\HH_1}^2&\leq 2\|f_\rho-\fbar\|_{\HH_1}^2+2\limsup\limits_{m\to\infty}\langle \fbar-f_{\zz(k),\la},f_\rho-\fbar\rangle_{\HH_1}\\ &=2\limsup\limits_{m\to\infty}\langle f_\rho-f_{\zz(k),\la},f_\rho-\fbar\rangle_{\HH_1}=0,
\end{align*}
which contradicts~(\ref{c3}). This completes the proof of the desired result~(\ref{consistency.rate}).
\end{proof}

\section{Proof of upper rates}\label{sec:proof-upper-rates} 
Here, we introduce some operators~$\Delta:=S_\xx A(f_\rho)-\yy$ and~$\Xi:=S_\xx(S_\xx^*S_\xx+\la I)^{-1}S_\xx^*$ used in the analysis of upper rates.

\begin{proof}[Proof of Theorem \ref{converge}]
By the definition of~$f_{\zz,\la}$ as a solution to the minimization problem~(\ref{fzl}), the inequality holds true:
\begin{equation*}
\norm{S_\xx A(f_{\zz,\la})-\yy}_m^2+\la\norm{f_{\zz,\la}-\fbar}_{\HH_1}^2 \leq \norm{S_\xx A(f_\rho)-\yy}_m^2+\la\norm{f_\rho-\fbar}_{\HH_1}^2
\end{equation*}
which implies
\begin{equation*}
\norm{S_\xx\brac{A(f_{\zz,\la})-A(f_\rho)}}_{m}^2+\la\norm{f_{\zz,\la}-f_\rho}_{\HH_1}^2 \leq 2\la\inner{f_\rho-\fbar,f_\rho-f_{\zz,\la}}_{\HH_1}+2\inner{A(f_\rho)-A(f_{\zz,\la}),S_\xx^*\Delta}_{\HH_2}.
\end{equation*}

Under the conditions (i) and (iii) of Assumption~\ref{A.assumption}, for~$f\in \mathcal{B}_d(f_\rho)$ we get
\begin{equation}\label{Taylor_exp_Rla}
A(f)=A(f_\rho)+A'(f_\rho)(f-f_\rho)+r(f)
\end{equation}
holds with
\begin{equation}\label{R_la}
\|I_Kr(f)\|_{\LL}\leq \frac{\gamma}{2}\|f-f_\rho\|_{\HH_1}^2 
\end{equation}
and
\begin{align}\label{Rla}
\norm{r(f)}_{\HH_2} =& \norm{A(f)-A(f_\rho)-A'(f_\rho)(f-f_\rho)}_{\HH_2}    = \norm{\int_{0}^1 \brac{A'\paren{f_\rho+t(f-f_\rho)}-A'(f_\rho)}\paren{f-f_\rho}dt}_{\HH_2}\\ \nonumber
\leq& \int_{0}^1\norm{\brac{A'\paren{f_\rho+t(f-f_\rho)}-A'(f_\rho)}}_{\HH_1\to\HH_2}\norm{f-f_\rho}_{\HH_1}dt \leq 2L\norm{f-f_\rho}_{\HH_1}.
\end{align}

Note that under condition~\eqref{l.la.condition.k}, from inequality~\eqref{c2} we get~$f_{\zz,\la}\in \mathcal{B}_d(f_\rho)$ with confidence~$1-\eta/2$ for~$d>4\norm{f_\rho-\fbar}_{\HH_1}$, therefore using the linearization of the non-linear operator~$A$ in \eqref{Taylor_exp_Rla} at~$f_{\zz,\la}$ and under Assumption~\ref{source.cond} we obtain,
\begin{align*}
&\norm{I_K\brac{A(f_{\zz,\la})-A(f_\rho)}}_{\LL}^2+\la\norm{f_{\zz,\la}-f_\rho}_{\HH_1}^2 \\
\leq & 2\la\inner{T^{1/2}g,f_\rho-f_{\zz,\la}}_{\HH_1}+2\inner{A(f_\rho)-A(f_{\zz,\la}),L_K(L_K+\la I)^{-1}S_\xx^*\Delta}_{\HH_2} \\
&+2\la\inner{A(f_\rho)-A(f_{\zz,\la}),(L_K+\la I)^{-1}S_\xx^*\Delta}_{\HH_2}+\inner{(L_K-S_\xx^*S_\xx)\brac{A(f_{\zz,\la})-A(f_\rho)},A(f_{\zz,\la})-A(f_\rho)}_{\HH_2}\\
\leq & 2\la\inner{g,T^{1/2}(f_\rho-f_{\zz,\la})}_{\HH_1}+2\mathcal{S}_e\norm{I_K\brac{A(f_{\zz,\la})-A(f_\rho)}}_{\LL}+2\sqrt{\la}L\mathcal{S}_e\norm{f_{\zz,\la}-f_\rho}_{\HH_1}\\
&+I_1\norm{A(f_{\zz,\la})-A(f_\rho)}_{\HH_2}^2\\
\leq & 2\la R\norm{I_K\brac{A(f_\rho)-A(f_{\zz,\la})+r(f_{\zz,\la})}}_{\LL}+2\mathcal{S}_e\norm{I_K\brac{A(f_{\zz,\la})-A(f_\rho)}}_{\LL}\\
&+2\sqrt{\la}L\mathcal{S}_e\norm{f_{\zz,\la}-f_\rho}_{\HH_1}+L^2I_1\norm{f_{\zz,\la}-f_\rho}_{\HH_1}^2\\
\leq & 2\paren{\la R+\mathcal{S}_e}\norm{I_K\brac{A(f_{\zz,\la})-A(f_\rho)}}_{\LL}+2\sqrt{\la}L\mathcal{S}_e\norm{f_{\zz,\la}-f_\rho}_{\HH_1}\\
&+L^2I_1\norm{f_{\zz,\la}-f_\rho}_{\HH_1}^2+\la\gamma R\norm{f_{\zz,\la}-f_\rho}_{\HH_1}^2,
\end{align*}
where~$I_1=\norm{S_\xx^*S_\xx-L_K}_{L(\HH_2)}$ and~$\mathcal{S}_e=\norm{(L_K+\la I)^{-1/2}S_\xx^*(S_\xx A(f_\rho)-\yy)}_{\HH_2}$.

It gives
\begin{align*}
&\paren{\norm{I_K\brac{A(f_{\zz,\la})-A(f_\rho)}}_{\LL}-\la R-\mathcal{S}_e}^2+\paren{\sqrt{\la\gamma_1}\norm{f_{\zz,\la}-f_\rho}_{\HH_1}-\frac{L}{\sqrt{\gamma_1}}\mathcal{S}_e}^2 \\
&\leq \paren{\la R+\mathcal{S}_e}^2+\frac{L^2}{\gamma_1}\mathcal{S}_e^2
\end{align*}
where~$\gamma_1=1-\gamma R-L^2I_1/\la$. This implies
\begin{equation*}
\norm{I_K\brac{A(f_{\zz,\la})-A(f_\rho)}}_{\LL}\leq 2R\la+\paren{2+\frac{L}{\sqrt{\gamma_1}}}\mathcal{S}_e
\end{equation*}
and
\begin{equation*}
\norm{f_{\zz,\la}-f_\rho}_{\HH_1}\leq \frac{1}{\sqrt{\gamma_1}}R\sqrt{\la}+\paren{\frac{1}{\sqrt{\gamma_1}}+\frac{2L}{\gamma_1}}\frac{\mathcal{S}_e}{\sqrt{\la}}.
\end{equation*}

Now under the Assumptions~\ref{fp}--\ref{assmpt1} using the estimates of Proposition~\ref{main.bound}, the inequality~(\ref{l.la.condition.k}) and~\eqref{smallness}, we obtain that~$\gamma_1=1/2-\gamma R>0$ and with the probability~$1-\eta$,
\begin{equation*}
\|I_K\brac{A(f_{\zz,\la})-A(f_\rho})\|_{\LL} \leq 2R \la + \frac{2(L+2\sqrt{\gamma_1})}{\sqrt{\gamma_1}}\left(\frac{\kappa M}{m\sqrt{\la}}+\sqrt{\frac{\Sigma^2\mathcal{N}(\la)}{m}}\right) \log\left(\frac{4}{\eta}\right)
\end{equation*}
and 
\begin{equation*}
\|f_{\zz,\la}-f_\rho\|_{\HH_1} \leq \frac{1}{\sqrt{\gamma_1}}R\sqrt{\la}+\frac{2(2L+\sqrt{\gamma_1})}{\gamma_1}\left(\frac{\kappa M}{m\la}+\sqrt{\frac{\Sigma^2\mathcal{N}(\la)}{m\la}}\right) \log\left(\frac{4}{\eta}\right).
\end{equation*}
which implies the desired result.
\end{proof}

For the analysis of Tikhonov regularization under general source condition, we consider the linearized and population
version (i.e. using theoretical expectation under~$\rho$) of the regularization scheme~(\ref{fzl}):
\begin{equation*}\label{fl.funl}
f_\la^l:=\mathop{\text{arg}\min}_{f \in \HH_1} \left\{\int_Z     \|A(f_\rho)(x)+A'(f_\rho)(f-f_\rho)(x)-y\|_Y^2d\rho(x,y)+\la\|f-\fbar\|_{\HH_1}^2\right\}.
\end{equation*}

Under Assumption~\ref{assmpt2}, using the fact~$\mathcal{E}(f):=\int_Z\|A(f_\rho)(x)+A'(f_\rho)(f-f_\rho)(x)-y\|_Y^2d\rho(x,y)=\|T^{1/2}(f-f_\rho)\|_{\HH_1}^2+\mathcal{E}(f_\rho)$, we get
\begin{equation}\label{fl}
f_\la^l = (T+\la I)^{-1}(Tf_\rho+\la \fbar).
\end{equation}

In the following proposition, we estimate the error bound of  approximation error~$f_\la^l-f_\rho$ which describes
the complexity of the true solution~$f_\rho$. The approximation
error is independent of the samples~$\zz$.
\begin{proposition}\label{approx.err}
Suppose Assumptions~\ref{fp},~\ref{source.cond} holds true. Then under the assumption that~$\phi(t)$ and~$t/{\phi(t)}$ are non-decreasing functions, we have
\begin{equation*}\label{fla.frho.K}
\norm{f_{\la}^l-f_\rho}_{\HH_1} \leq R\phi(\la).
\end{equation*}
\end{proposition}
\begin{proof}
From the definition of~$f_\la^l$ in~\eqref{fl} and Assumption~\ref{source.cond} we get,
$$f_\rho-f_\la^l=\la(T+\la I)^{-1}\phi(T)g.$$
Under the assumption that~$\phi(t)$ and~$t/{\phi(t)}$ are non-decreasing functions, we obtain,
$$\norm{f_{\la}^l-f_\rho}_{\HH_1}\leq R\phi(\la).$$
\end{proof}

Under Assumption~\ref{source.cond} from Proposition~\ref{approx.err}, we observe that~$f_\la^l\in\mathcal{D}(A)\cap\mathcal{B}_d(f_\rho)$, provided~$\la$ is sufficiently small. 

In the following theorem, we estimate the quantity~$f_{\la}^l-f_{\zz,\la}$ and use the bound of approximation error from the above proposition to find the error bound for~$f_\rho-f_{\zz,\la}$. 
\begin{proof}[Proof of Theorem \ref{convergence}]
The main idea of the proof is to compare~$f_{\zz,\la}$ and~$f_{\la}^l$. From the definition of~$f_{\zz,\la}$ in~\eqref{fzl}, we have
\begin{equation}\label{idea}
 \norm{S_\xx A(f_{\zz,\la})-\yy}^2_m +\la \|f_{\zz,\la}-\fbar\|_{\HH_1}^2\leq  \norm{S_\xx A(f_{\la}^l)-\yy}^2_m+\la \|f_{\la}^l-\fbar\|_{\HH_1}^2.
\end{equation}

Using the linearization of operator~$A$ in~(\ref{Taylor_exp_Rla}) we reexpress the inequality~(\ref{idea}) as follows,
\begin{align*}
\|f_{\zz,\la}-f_{\la}^l\|_{\HH_1}^2&\leq 2\langle f_{\la}^l-f_{\zz,\la},f_{\la}^l-\fbar\rangle_{\HH_1}+\frac{1}{\la}\left\{\|B_\xx (f_{\la}^l-f_\rho)+\Delta+S_\xx(r(f_{\la}^l))\|_m^2 \right.   \\
&\left.- \|B_\xx (f_{\zz,\la}-f_\rho)+\Delta+S_\xx(r(f_{\zz,\la}))\|_m^2\right\}
\end{align*}

Now we decompose the second and third term in the right hand side as follows:
\begin{align*}
\|f_{\zz,\la}-f_{\la}^l\|_{\HH_1}^2 &\leq 2\langle f_{\la}^l-f_{\zz,\la},f_{\la}^l-\fbar\rangle_{\HH_1}+\frac{1}{\la}
\left\{\|\Xi\Delta+S_\xx(r(f_{\la}^l))\|_m^2\right.\\ \nonumber
&\qquad \left.+2\langle \Xi\Delta+S_\xx(r(f_{\la}^l)), B_\xx(f_{\la}^l-f_\rho)+(I-\Xi)\Delta\rangle_m\right.\\ \nonumber
&\qquad \left. -\|B_\xx(f_{\zz,\la}-f_\la^l)+\Xi\Delta+S_\xx(r(f_{\zz,\la}))\|_m^2\right.\\ \nonumber
&\qquad \left.-2\langle B_\xx(f_{\zz,\la}-f_\la^l)+\Xi\Delta+S_\xx(r(f_{\zz,\la})),B_\xx(f_{\la}^l-f_\rho)+(I-\Xi)\Delta\rangle_m\right\}.
\end{align*}

The fourth term in the right hand side is negative, therefore it can be ignored,
leading to:
\begin{align}\label{Est}
  \|f_{\zz,\la}-f_{\la}^l\|_{\HH_1}^2
&\leq \frac{1}{\la}\left\{2\langle f_{\la}^l-f_{\zz,\la},\la (f_{\la}^l-\fbar)+T_\xx(f_\la^l- f_\rho)+B_\xx^*(I-\Xi)\Delta\rangle_{\HH_1}\right.\\ \nonumber
&\qquad +2\langle r(f_{\la}^l)-r(f_{\zz,\la}),S_\xx^*B_\xx (f_{\la}^l-f_\rho)+S_\xx^*(I-\Xi)\Delta\rangle_{\HH_2}\\ \nonumber
&\qquad +\|\Xi\Delta+S_\xx(r(f_{\la}^l))\|_m^2\rbrace.
\end{align}

The definition of~$f_\la^l$ in~\eqref{fl} implies that
\begin{equation}\label{fll}
\la (f_\la^l-\fbar)=T(f_\rho-f_\la^l).
\end{equation}

Therefore, from inequality~(\ref{Est}),  using Assumption~\ref{A.assumption} (ii) and~\eqref{fll} we get:
\begin{align}\label{Estimate}
\|f_{\zz,\la}-f_{\la}^l\|_{\HH_1}^2 \leq & \frac{1}{\la}\left\{2\langle f_{\la}^l-f_{\zz,\la},(T_\xx-T)(f_\la^l- f_\rho)+B_\xx^*(I-\Xi)\Delta\rangle_{\HH_1}\right.\\ \nonumber
&+2\langle r(f_{\la}^l)-r(f_{\zz,\la}),S_\xx^*B_\xx (f_{\la}^l-f_\rho)+S_\xx^*(I-\Xi)\Delta\rangle_{\HH_2}\\ \nonumber
&+2\|\Xi\Delta\|_m^2+2\|S_\xx(r(f_{\la}^l))\|_m^2\rbrace\\ \nonumber
\leq& \frac{2}{\la}\left\{\langle f_{\la}^l-f_{\zz,\la},(T_\xx-T)(f_\la^l- f_\rho)+\la A'(f_\rho)^*(S_\xx^*S_\xx+\la I)^{-1}S_\xx^*\Delta\rangle_{\HH_1}\right.\\ \nonumber
&+\langle r(f_{\la}^l)-r(f_{\zz,\la}),(S_\xx^*S_\xx-L_K) A'(f_\rho)(f_{\la}^l-f_\rho)+\la (S_\xx^*S_\xx+\la I)^{-1}S_\xx^*\Delta\rangle_{\HH_2} \\ \nonumber
&+\langle I_K \{ r(f_{\la}^l)-r(f_{\zz,\la})\},B(f_{\la}^l-f_\rho)\rangle_{\mathscr{L}^2(X,\rho_X;Y)}+\|\Xi\Delta\|_m^2\\ \nonumber
&+\|I_Kr(f_{\la}^l)\|_{\mathscr{L}^2(X,\rho_X;Y)}^2+\langle(S_\xx^*S_\xx-L_K)r(f_{\la}^l),r(f_{\la}^l)\rangle_{\HH_2}\rbrace  \\ \nonumber
\leq& \frac{2}{\la}\left\{\norm{f_{\zz,\la}-f_{\la}^l}_{\HH_1} \left(L^2I_1\mathcal{A}_e+\sqrt{\la}L I_2\mathcal{S}_e\right)+\norm{r(f_{\la}^l)-r(f_{\zz,\la})}_{\HH_2}\left(L I_1\mathcal{A}_e+\sqrt{\la} I_2\mathcal{S}_e\right)\right.\\ \nonumber
&+\| I_K \{ r(f_{\la}^l)-r(f_{\zz,\la})\}\|_{\mathscr{L}^2(X,\rho_X;Y)} \|B(f_{\la}^l-f_\rho)\|_{\mathscr{L}^2(X,\rho_X;Y)}+I_2^2\mathcal{S}_e^2\\  \nonumber
&\left.+\|I_K r(f_{\la}^l)\|_{\mathscr{L}^2(X,\rho_X;Y)}^2+I_1\|r(f_{\la}^l)\|_{\HH_2}^2\right\},
\end{align}
where~$\mathcal{A}_e=\norm{f_{\la}^l-f_\rho}_{\HH_1}$,~$\mathcal{S}_e=\norm{(L_K+\la I)^{-1/2}S_\xx^*(S_\xx A(f_\rho)-\yy)}_{\HH_2}$, ~$I_1=\|S_\xx^*S_\xx-L_K\|_{\mathcal{L}(\HH_2)}$ and~$I_2=\|(S_\xx^*S_\xx+\la I)^{-1/2}(L_K+\la I)^{1/2}\|_{\mathcal{L}(\HH_2)}$.

We have~$\norm{B f}_{\mathscr{L}^2(X,\rho_X;Y)} \leq\norm{(T+\la I)^{1/2}f}_{\HH_1}$,~$f\in\mathcal{D}(A)\subset\HH_1$, therefore we obtain,
\begin{equation}\label{Rl}
\norm{B(f_\rho-f_\la^l)}_{\mathscr{L}^2(X,\rho_X;Y)} = \la\norm{ B(T+\la I)^{-1}T^{1/2} w}_{\mathscr{L}^2(X,\rho_X;Y)} \leq \la\norm{w}_{\HH_1}. 
\end{equation}

Using the inequalities~(\ref{R_la}),~\eqref{Rla},~(\ref{Rl}) in~(\ref{Estimate}) we obtain,
\begin{eqnarray*}
\|f_{\zz,\la}-f_{\la}^l\|_{\HH_1}^2  &\leq& \frac{2}{\la}\left\lbrace\la\gamma\norm{w}_{\HH_1}\norm{f_{\zz,\la}-f_\la^l}_{\HH_1}^2+\delta_1\|f_{\zz,\la}-f_{\la}^l\|_{\HH_1}+\delta_2 \right\rbrace,
\end{eqnarray*}
where~$\delta_1=3L^2I_1\mathcal{A}_e+3\sqrt{\la}LI_2 \mathcal{S}_e$ and~$\delta_2=I_2^2\mathcal{S}_e^2+4\sqrt{\la}L I_2\mathcal{A}_e\mathcal{S}_e+\gamma^2\mathcal{A}_e^4/4+3\gamma\la\norm{w}_{\HH_1}\mathcal{A}_e^2/2+8L^2I_1\mathcal{A}_e^2$.

Under the condition~(\ref{smallness}) we have,
\begin{equation*}
\|f_{\zz,\la}-f_{\la}^l\|_{\HH_1}^2 \leq \frac{2}{\gamma_2\la}\left\lbrace\delta_1\|f_{\zz,\la}-f_{\la}^l\|_{\HH_1}+\delta_2 \right\rbrace,
\end{equation*}
where~$\gamma_2=1-2\gamma\norm{w}_{\HH_1}$.

We have,
\begin{equation*}\label{fzl-fz}
\left(\|f_{\zz,\la}-f_{\la}^l\|_{\HH_1}-\frac{\delta_1}{\la\gamma_2}\right)^2 \leq \frac{\delta_1^2}{\la^2\gamma_2^2}+\frac{2\delta_2}{\la\gamma_2},
\end{equation*}
which implies
\begin{align*}
\|f_{\zz,\la}-f_{\la}^l\|_{\HH_1} 
&\leq \frac{2\delta_1}{\la\gamma_2}+\sqrt{\frac{2\delta_2}{\la\gamma_2}}.
\end{align*}

Using the triangle inequality~$\norm{f_{\zz,\la}-f_\rho}_{\HH_1}\leq\norm{f_{\zz,\la}-f_{\la}^l}_{\HH_1}+\norm{f_{\la}^l-f_\rho}_{\HH_1}$ we obtain,
\begin{align*}
\|f_{\zz,\la}-f_\rho\|_{\HH_1} 
&\leq  \brac{c_1+c_2\frac{2I_1}{\la}+c_3\sqrt{\frac{2I_1}{\la}}+c_4\sqrt{\frac{I_2}{2}}}\mathcal{A}_e+\brac{c_5\frac{I_2}{\sqrt{2}}+c_6\sqrt{\frac{I_2}{2}}}\left(\frac{\mathcal{S}_e}{2\sqrt{\la}}\right),
\end{align*}
where~$c_1=1+\gamma\norm{w}_{\HH_1}/\sqrt{2\gamma_2}+\sqrt{3\gamma\norm{w}_{\HH_1}/\gamma_2}$,~$c_2=3L^2/\gamma_2$,~$c_3=\sqrt{8L^2/\gamma_2}$,~$c_4=\sqrt{4L/\gamma_2}$,~$c_5=2\sqrt{2/\gamma_2}+12\sqrt{2}L/\gamma_2$ and~$c_6=2\sqrt{4L/\gamma_2}$.

Now using the estimate of Proposition~\ref{I1} with the inequality~(\ref{l.la.condition.k}), we get with the probability~$1-\eta/2$,
\begin{equation*}
\|f_{\zz,\la}-f_\rho\|_{\HH_1} \leq (c_1+c_2+c_3+c_4) \mathcal{A}_e+(c_5+c_6)\left(\frac{\mathcal{S}_e}{2\sqrt{\la}}\right).
\end{equation*}

Under Assumptions~\ref{fp}--\ref{assmpt1},~\ref{source.cond} from Proposition~\ref{main.bound},~\ref{approx.err}, we obtain with the confidence~$1-\eta$,
\begin{equation*}
\|f_{\zz,\la}-f_\rho\|_{\HH_1} \leq (c_1+c_2+c_3+c_4) R\phi(\la)+(c_5+c_6)\left(\frac{\kappa M}{m\la}+\sqrt{\frac{\Sigma^2\mathcal{N}(\la)}{m\la}}\right) \log\left(\frac{4}{\eta}\right).
\end{equation*}
which implies the desired result.
\end{proof}

\begin{proof}[Proof of Theorem \ref{err.upper.bound.p.para}]

\begin{enumerate}[(i)] 
\item Under the parameter choice~$\la=\Theta^{-1}\paren{m^{-1/2}}$ we have
$$\frac{1}{m\la}\leq\frac{\phi(\la)}{\sqrt{m}}.$$

From Theorem~\ref{convergence} and the bound~(\ref{N(l).bd}), it follows that with the confidence~$1-\eta$,
\begin{equation}\label{fzl.fp.Theta.p}
\|f_{\zz,\la}-f_\rho\|_{\HH_1} \leq C' \phi\left(\Theta^{-1}\paren{m^{-1/2}}\right)\log\left(\frac{4}{\eta}\right).
\end{equation}
where~$C':=(c_1+c_2+c_3+c_4+c_5+c_6)(R+\kappa M+\kappa L\Sigma).$

Now defining~$\tau:=C'\log\left(\frac{4}{\eta}\right)$ gives
$$\eta=\eta_{\tau}=4e^{-\tau/C'}.$$
The estimate~(\ref{fzl.fp.Theta.p}) can be reexpressed as
\begin{equation}\label{minimax.p2}
\mathbb{P}_{\zz\in Z^m}\left\{\|f_{\zz,\la}-f_\rho\|_{\HH_1}>\tau R \phi\left(\Theta^{-1}\paren{m^{-1/2}}\right)\right\}\leq \eta_{\tau}.
\end{equation}
\item 
From the condition \eqref{l.la.condition.k} we have~$8\kappa^2\leq \sqrt{m}\la$. This together with the parameter choice~$\la=\Psi^{-1}\paren{m^{-1/2}}$ implies that
$$\frac{1}{m\la}=\frac{\la^{-\frac{1}{2}+\frac{1}{2b}}\phi(\la)}{\sqrt{m}}\leq \frac{\la^{\frac{1}{2}+\frac{1}{2b}}\phi(\la)}{8\kappa^2}.$$
Now for~$\la\geq 1$ and~$b>1$ we have~$\la^{-\frac{1}{2}+\frac{1}{2b}}\leq 1$, therefore~$\frac{1}{m\la}\leq \phi(\la)$. On the other hand, for~$\la \leq 1$ we have~$\la^{\frac{1}{2}+\frac{1}{2b}}\leq 1$, therefore~$\frac{1}{m\la}\leq \frac{\phi(\la)}{8\kappa^2}$. Hence, from Theorem~\ref{convergence} and the inequality~(\ref{N(l).bound}), it follows that with the confidence~$1-\eta$,
\begin{equation}\label{fzl.fl.inter}
\|f_{\zz,\la}-f_\rho\|_{\HH_1}\leq C'' \phi\left(\Psi^{-1}\paren{m^{-1/2}}\right)\log\left(\frac{4}{\eta}\right),
\end{equation}
where~$C'':=(c_1+c_2+c_3+c_4+c_5+c_6)(R+\kappa M\max(1,\frac{1}{8\kappa^2})+\Sigma \sqrt{C_{\beta,b}}).$

Now defining~$\tau:=C''\log\left(\frac{4}{\eta}\right)$ gives
$$\eta=\eta_\tau=4e^{-\tau/C''}.$$
The estimate~(\ref{fzl.fl.inter}) can be reexpressed as
\begin{equation}\label{minimax.p1}
\mathbb{P}_{\zz\in Z^m}\left\{\|f_{\zz,\la}-f_\rho\|_{\HH_1}>\tau R\phi\left(\Psi^{-1}\paren{m^{-1/2}}\right)\right\} \leq \eta_{\tau}.
\end{equation}
\end{enumerate}
Then from~(\ref{minimax.p2}) and~(\ref{minimax.p1}) our conclusions follow.
\end{proof}

%
%
%
%
%

\section{Proof of lower rates}\label{sec:proof-lower-rates}

The following proposition is a variant of Proposition~4 \cite{Caponnetto} for the non-linear statistical inverse problem.
\begin{proposition}
For the probability measure~$\rho_f$, defined in~(\ref{p.f}), parameterized by~$f \in \mathcal{D}(A)\subset\HH_1$:
\begin{enumerate}[(i)]
\item The solution~$f_\rho$ for the probability measure~$\rho=\rho_f$ is~$f$. 
\item  The probability measure~$\rho_f$ satisfies Assumption~\ref{noise.cond} provided that
\begin{equation}\label{condition}
dJ+J/4\leq M \text{ and }2dJ\leq\Sigma.
\end{equation}
\end{enumerate}
\end{proposition}
\begin{proof}
The first point can be easily observed. Now we check the condition on the probability measure~$\rho_f$ for the second point. 
 
Under the condition~\eqref{condition} for the conditional probability measure~$\rho_f(y|x)$ we have,
\begin{eqnarray*}
&&\int_Y\left(e^{\|y-A(f)(x)\|_Y/M}-\frac{\|y-A(f)(x)\|_Y}{M}-1\right)d\rho_f(y|x) \\ \nonumber
&\leq& \int_Y\|y-A(f)(x)\|_Y^2d\rho_f(y|x)\sum\limits_{i=2}^\infty \frac{(dJ+\|A(f)(x)\|_Y)^{i-2}}{M^ii!}   \\   \nonumber
&\leq& 2d^2J^2\sum\limits_{i=2}^\infty \frac{(dJ+\|A(f)(x)\|_Y)^{i-2}}{M^ii!}\leq\frac{\Sigma^2}{2M^2}
\end{eqnarray*}
which implies that for the solution~$f_\rho=f$ the probability measure~$\rho_f$ satisfies Assumption~\ref{noise.cond}.  
\end{proof} 

\begin{proposition}\label{fi.fj.kull}
Under Assumptions~\ref{assmpt1},~\ref{A.assumption1}, there is an~$\varepsilon_0 > 0$ such that for all~$0 < \varepsilon \leq \varepsilon_0$, there exists~$N_\varepsilon\in \NN$ and each~$f_1,\ldots,f_{N_\varepsilon} \in \HH_1$ (depending on~$\varepsilon$) satisfying:
\begin{enumerate}[(i)]


\item For~$i = 1, \ldots , N_\varepsilon$,~$f_i\in\Omega(\rho_{f_i},\phi,R)$  and for any~$i, j = 1, \ldots , N_\varepsilon$ with~$i \neq j$,
\begin{equation*}\label{fi.fj.cond}
\varepsilon\upsilon \leq \|f_i - f_j\|_{\HH_1},
\end{equation*}
where~$\upsilon=1-\|I-R_{f_i}\|_{\mathcal{L}(\HH_1)}-\|I-R_{f_j}\|_{\mathcal{L}(\HH_1)}$ is positive for sufficiently small~$\varepsilon$ and~$R_{f_i}$ are defined in Assumption~\ref{A.assumption1} (iv).

\item Let~$\rho_i := \rho_{f_i}$,~$\rho_j := \rho_{f_j}$ be given by~(\ref{p.f}) for~$f_i\in\Omega(\rho_i,\phi,R)$ and~$f_j\in\Omega(\rho_j,\phi,R)$,~$i, j = 1, \ldots , N_\varepsilon$, then the Kullback–Leibler information~$K(\rho_{f_i},\rho_{f_j})$ fulfills the inequality:
\begin{equation}\label{Kull_lieb}
\mathcal{K}(\rho_{f_i},\rho_{f_j}) \leq \frac{16}{15dJ^2}\|I_K\{A(f_i)-A(f_j)\}\|_{\LL}^2.
\end{equation} 

Further, it holds
\begin{equation}\label{kullback}
\mathcal{K}(\rho_i,\rho_j) \leq \widetilde{C}\left(\frac{\varepsilon^2}{\ell_\varepsilon^b}+\varepsilon^4\right),
\end{equation}
where~$N_\varepsilon \geq e^{\ell_\varepsilon/24}$ for~$\ell_\varepsilon=\left\lfloor\frac{1}{2}\left(\frac{\al}{\phi^{-1}(\varepsilon/R)}\right)^{1/b}\right\rfloor$ and ~$\widetilde{C}=\frac{16 c''}{15dJ^2}$.

 \item The eigenvalues~$(t_n^i)_{n\in\NN}$ of the operators~$T_i=A'(f_i)^*I_K^*I_KA'(f_i)$ follow the polynomial decay for the each~$f_i$ ($1\leq i\leq N_\varepsilon$): For fixed positive constants~$\al_i,\beta_i$ and~$b>1$,
\begin{equation*}
\al_i n^{-b}\leq t_n^i\leq\beta_i n^{-b}~~\forall n\in\NN.
\end{equation*}
\end{enumerate}
\end{proposition}
\begin{proof}
For the initial guess~$\fbar$ of the solution of the functional~\eqref{fzl}, let~$(e_n)_{n\in\NN}$ be an orthonormal basis of the Hilbert space~$\HH_1$ of eigenvectors of the operator~$\overline{T}=A'(\fbar)^*I_K^*I_KA'(\fbar)$ corresponding to the eigenvalues~$(t_n)_{n\in\NN}$. For given~$\varepsilon>0$, we define
$$g=\sum\limits_{n=\ell+1}^{2\ell}\frac{\varepsilon\pi^{n-\ell}e_n}{\sqrt{\ell}\phi(t_n)},$$
where~$\pi=(\pi^1,\ldots,\pi^{\ell})\in \{-1,+1\}^{\ell}$.

Under the polynomial decay condition~$\al \leq n^bt_n$ on the eigenvalues of the operator~$\overline{T}$, we get
$$\|g\|_{\HH_1}^2=\sum\limits_{n=\ell+1}^{2\ell}\frac{\varepsilon^2}{\ell\phi^2(t_n)} \leq \sum\limits_{n=\ell+1}^{2\ell}\frac{\varepsilon^2}{\ell\phi^2\left(\frac{\al}{n^b}\right)} \leq \frac{\varepsilon^2}{\phi^2\left(\frac{\al}{2^b\ell^b}\right)}\leq R^2,$$
for 
\begin{equation}\label{ell.bound}
\ell=\ell_\varepsilon=\left\lfloor\frac{1}{2}\left(\frac{\al}{\phi^{-1}(\varepsilon/R)}\right)^{1/b}\right\rfloor,
\end{equation} 
where~$\lfloor x \rfloor$ is the greatest integer less than or equal
to~$x$. 

We choose~$\varepsilon_\circ$ such that~$\ell_{\varepsilon_\circ}>16$. Then from Proposition~6 \cite{Caponnetto}, for every positive~$\varepsilon<\varepsilon_\circ~(\ell_\varepsilon>\ell_{\varepsilon_\circ})$ there exists an integer~$N_\varepsilon \in \NN$ and~$\pi_1,\ldots,\pi_{N_\varepsilon} \in \{-1,+1\}^{\ell_\varepsilon}$ such that
for all~$1 \leq i,j \leq N_\varepsilon,~i \neq j$ it holds

\begin{equation}\label{sigma.i.j}
\sum_{n=1}^{\ell_\varepsilon}(\pi_i^n-\pi_j^n)^2 \geq \ell_{\varepsilon}
\end{equation}
and
\begin{equation}\label{N.epsilon.k}
\log(N_\varepsilon) \geq \ell_\varepsilon/24.
\end{equation}

Now we construct~$N_\varepsilon$-vectors satisfying the source condition (Assumption~\ref{source.cond}). For~$\varepsilon$ such that~$0<\varepsilon<\varepsilon_\circ$, we define 
\begin{equation}\label{gi}
g_i=\sum\limits_{n=\ell_\varepsilon+1}^{2\ell_\varepsilon}\frac{\varepsilon\pi_i^{n-\ell_\varepsilon}e_n}{\sqrt{\ell_\varepsilon}\phi(t_n)},
\end{equation}
where~$\pi_i=(\pi_i^1,\ldots,\pi_i^{\ell_\varepsilon})\in \{-1,+1\}^{\ell_\varepsilon}$ for~$i=1,\ldots,N_\varepsilon$. Hence from~\eqref{ell.bound}, we observe that~$\norm{g_i}_{\HH_1}\leq R$.

Suppose~$F(f)=\phi(T)g+\fbar$ for~$B=I_KA'(f)$,~$T=B^*B$ and some ~$g\in \HH_1$, then from Assumptions~\ref{assmpt1},~\ref{A.assumption1} (iv) for the Lipschitz continuous function~$\theta(t)=\phi(t^2)$ from Propositions~\ref{OP_Lip},~\ref{TB} under the Lipschitz continuity of the Fr{\'e}chet derivative of the operator~$A$ we obtain,
\begin{align*}
\norm{F(\tilde{f})-F(f)}_{\HH_1} &\leq  \norm{\{\phi(\widetilde{T})-\phi(T)\}g}_{\HH_1}\leq \norm{\{\theta(\widetilde{T}^{1/2})-\theta(T^{1/2})\}g}_{\HH_1} \\ \nonumber
&\leq \norm{g}_{\HH_1}
\norm{\theta(\widetilde{T}^{1/2})-\theta(T^{1/2})}_{\mathcal{L}(\HH_1)} \leq \norm{g}_{\HH_1}
\norm{\theta(\widetilde{T}^{1/2})-\theta(T^{1/2})}_{HS}  \\
&\leq L_\theta\norm{g}_{\HH_1} \norm{\widetilde{T}^{1/2}-T^{1/2}}_{HS} \leq\sqrt{2}L_\theta\norm{g}_{\HH_1}
\norm{\widetilde{B}-B}_{HS} \\
&\leq \sqrt{2}L_\theta\norm{g}_{\HH_1}
\norm{I_K\brac{A'(\tilde{f})-A'(f)}}_{HS} \leq \gamma\sqrt{2}L_\theta\norm{g}_{\HH_1}
\norm{\tilde{f}-f}_{\HH_1},
\end{align*}
where~$\widetilde{B}=I_K A'(\tilde{f})$ and~$\widetilde{T}=\widetilde{B}^*\widetilde{B}$.

If~$\gamma\sqrt{2}L_\theta\norm{g}_{\HH_1}<1$, then~$F$ is a contraction map. Hence, there exists a fixed point~$f_*\in\HH_1$ such that 
\begin{equation}\label{fix.point}
f_*=F(f_*)=\phi(T_*)g+\fbar,
\end{equation}
where~$T_*=(I_KA'(f_*))^*I_KA'(f_*)$. 
%

Hence for each~$g_i$ defined in~\eqref{gi} from~\eqref{fix.point} there exist~$f_i$~$(1 \leq i \leq N_\varepsilon)$ such that  

\begin{equation*}
f_i-\fbar=\phi(T_i)g_i, 
\end{equation*} 
where~$B_i=I_KA'(f_i)$ and~$T_i=B_i^*B_i$, i.e.,~$f_i\in\Omega(\rho_{f_i},\phi,R)$ provided that~$\gamma\sqrt{2}L_\theta\norm{g_i}_{\HH_1}<1$ for~$1\leq i\leq N_\varepsilon$ which can be satisfied by making the quantity~$\norm{g_i}_{\HH_1}$ arbitrarily small as~$\varepsilon\to 0$.

Under Assumption~\ref{A.assumption1} (iv) from eqn.~\eqref{gi}  for all~$1 \leq i,j \leq N_\varepsilon,$ we get,
\begin{equation*}
f_i-\fbar = \phi(T_i)g_i=R_{f_i}\phi(\overline{T})g_i 
= \left\lbrace I-(I-R_{f_i})\right\rbrace\sum\limits_{n=\ell_\varepsilon+1}^{2\ell_\varepsilon}\frac{\varepsilon \pi_i^{n-\ell_\varepsilon}e_n}{\sqrt{\ell_\varepsilon}}
\end{equation*}
and
\begin{equation}\label{fifj}
f_i-f_j = \sum\limits_{n=\ell_\varepsilon+1}^{2\ell_\varepsilon}\frac{\varepsilon (\pi_i^{n-\ell_\varepsilon}-\pi_j^{n-\ell_\varepsilon})e_n}{\sqrt{\ell_\varepsilon}}-(I-R_{f_i})\sum\limits_{n=\ell_\varepsilon+1}^{2\ell_\varepsilon}\frac{\varepsilon \pi_i^{n-\ell_\varepsilon}e_n}{\sqrt{\ell_\varepsilon}}+(I-R_{f_j})\sum\limits_{n=\ell_\varepsilon+1}^{2\ell_\varepsilon}\frac{\varepsilon \pi_j^{n-\ell_\varepsilon}e_n}{\sqrt{\ell_\varepsilon}}
\end{equation}
which implies from~(\ref{sigma.i.j}) that
\begin{equation}\label{fi_fs}
\|f_i-\fbar\|_{\HH_1}\leq \varepsilon(1+\|I-R_{f_i}\|_{\mathcal{L}(\HH_1)})
\end{equation}
and
\begin{equation}\label{f.i.j.bound.l}
\varepsilon\upsilon\leq\|f_i-f_j\|_{\HH_1}
\end{equation}
where~$\upsilon=1-\|I-R_{f_i}\|_{\mathcal{L}(\HH_1)}-\|I-R_{f_j}\|_{\mathcal{L}(\HH_1)}$.

Then under Assumption~\ref{A.assumption1} (iv) and~\eqref{fi_fs} we have,
$$\|I-R_f\|_{\mathcal{L}(\HH_1)}\leq \zeta\norm{f-\fbar}_{\HH_1}\leq \zeta\varepsilon(1+\|I-R_f\|_{\mathcal{L}(\HH_1)})$$
which implies that
$$\|I-R_f\|_{\mathcal{L}(\HH_1)}\leq\frac{\zeta\varepsilon}{1-\zeta\varepsilon}.$$

From Assumptions~\ref{assmpt1},~\ref{A.assumption1} (ii) and~\eqref{fifj} we get,
\begin{equation*}
\|\overline{B}(f_i-f_j)\|_{\mathscr{L}^2(X,\nu;Y)} \leq \|\overline{B}\phi(\overline{T})(g_i-g_j)\|_{\mathscr{L}^2(X,\nu;Y)}+\varepsilon\kappa L(\|I-R_{f_i}\|_{\mathcal{L}(\HH_1)}+\|I-R_{f_j}\|_{\mathcal{L}(\HH_1)}),
\end{equation*}
where~$\overline{B}=I_K\circ(A'(\fbar))$.

Now from Assumptions~\ref{A.assumption1} (iv), (v) and~\eqref{fi_fs} we get, 
\begin{align}\label{Bfij}
\|\overline{B}(f_i-f_j)\|_{\mathscr{L}^2(X,\nu;Y)}
&\leq \left(\sum_{n=\ell_\varepsilon+1}^{2\ell_\varepsilon} \frac{t_n\varepsilon^2\left(\pi_{i}^{n-\ell_\varepsilon}-\pi_{j}^{n-\ell_\varepsilon}\right)^2}{\ell_\varepsilon}\right)^{1/2}+\varepsilon \kappa L\zeta(\|f_i-\fbar\|_{\HH_1}+\|f_j-\fbar\|_{\HH_1})\\  \nonumber
&\leq \left(\sum_{n=\ell_\varepsilon+1}^{2\ell_\varepsilon} \frac{\beta\varepsilon^2\left(\pi_{i}^{n-\ell_\varepsilon}-\pi_{j}^{n-\ell_\varepsilon}\right)^2}{\ell_\varepsilon n^b}\right)^{1/2}+c\varepsilon^2\\ \nonumber
&\leq \left(\sum\limits_{n=\ell_\varepsilon+1}^{2\ell_\varepsilon}\frac{4\beta\varepsilon^2}{\ell_\varepsilon n^b}\right)^{1/2}+c\varepsilon^2
\leq \left(\frac{4\beta\varepsilon^2}{\ell_\varepsilon}\int_{\ell_\varepsilon}^{2\ell_\varepsilon}\frac{1}{x^b}dx\right)^{1/2}+c\varepsilon^2 \leq c'\frac{\varepsilon}{\ell_\varepsilon^{b/2}}+c\varepsilon^2,
\end{align}
where~$c=\kappa L\zeta(2+\|I-R_{f_i}\|_{\mathcal{L}(\HH_1)}+\|I-R_{f_j}\|_{\mathcal{L}(\HH_1)})$  and~$c'=\left(\frac{4\beta}{(b-1)}\left(1-\frac{1}{2^{b-1}}\right)\right)^{1/2}$.

Note that the Lipschitz continuity of the Fr{\'e}chet derivative of the operator~$A$ (Assumption~\ref{A.assumption1} (iii)) imply that
\begin{equation*}\label{Taylor_exp_Rl}
A(f_i)=A(\fbar)+A'(\fbar)(f_i-\fbar)+r(f_i)
\end{equation*}
holds with
\begin{equation*}\label{R_l}
\|I_Kr(f_i)\|_{\LL}\leq \frac{\gamma}{2}\|f_i-\fbar\|_{\HH_1}^2.
\end{equation*}

Hence, for~$1\leq i,j \leq N_\varepsilon$, from the inequality~\eqref{fi_fs},~\eqref{Bfij} we have,
\begin{eqnarray}\label{f.i.j.bound.u}
\|I_K\{A(f_i)-A(f_j)\}\|_{\mathscr{L}^2(X,\nu;Y)}^2 \hspace{-3mm}&=&\hspace{-3mm} \left\{\|\overline{B}(f_i-f_j)\|_{\mathscr{L}^2(X,\nu;Y)}+\|I_K \{r(f_i)-r(f_j)\}\|_{\mathscr{L}^2(X,\nu;Y)}\right\}^2 \\ \nonumber
\hspace{-3mm}&\leq&\hspace{-3mm} 2\|\overline{B}(f_i-f_j)\|_{\mathscr{L}^2(X,\nu;Y)}^2 + 2\|I_K\brac{r(f_i)-r(f_j)}\|_{\LL}^2 \\ \nonumber
\hspace{-3mm}&\leq&\hspace{-3mm} 2\|\overline{B}(f_i-f_j)\|_{\mathscr{L}^2(X,\nu;Y)}^2+ \gamma^2\|f_i-\fbar\|_{\HH_1}^4+\gamma^2\|f_j-\fbar\|_{\HH_1}^4   \\ \nonumber
&\leq& c''\left(\frac{\varepsilon^2}{\ell_\varepsilon^b}+\varepsilon^4\right),
\end{eqnarray}
where~$c''=4c^2+4c'^2+\gamma^2\{(1+\|I-R_{f_i}\|_{\mathcal{L}(\HH_1)})^4+(1+\|I-R_{f_j}\|_{\mathcal{L}(\HH_1)})^4\}$. 

Under Assumption~\ref{A.assumption1} (iv), if~$\norm{ f - \fbar}_{\HH_1} < 1/\zeta$, then from Neumann series we have that~$\norm{R_{f}^{-1}}_{\mathcal{L}(\HH_1)}<\infty$. Therefore, 
$$ \phi(T) = R_{f} \phi(T_{\ast}) \quad \text{and} \quad\phi(T_{\ast}) = R_{f}^{-1} \phi(T).~$$
Now using the relation for singular values~$s_{j}(A B) \leq \norm{A}{} s_{j}(B)$ for~$j\in\NN$ (see Chapter 11 \cite{Pietsch}) we obtain,
\begin{equation}\label{sn1}
\phi(s_{j}(T))=s_{j}(\phi(T)) \leq \norm{R_{f}}_{\mathcal{L}(\HH_1)}    s_{j}(\phi(T_{\ast}))=\norm{R_{f}}_{\mathcal{L}(\HH_1)}    \phi(s_{j}(T_{\ast}))
\end{equation}  
and
\begin{equation}\label{sn2}
\phi(s_{j}(T_{\ast}))=  s_{j}(\phi(T_{\ast})) \leq \norm{R_{f}^{-1}}_{\mathcal{L}(\HH_1)} s_{j}(\phi(T))= \norm{R_{f}^{-1}}_{\mathcal{L}(\HH_1)} s_{j}(\phi(T))
\end{equation}

Consequently, for small enough~$\norm{f_i - \fbar}_{\HH_1}$ corresponding to small~$\varepsilon$, the
eigenvalues of~$T_i$ and~$T_{\ast}$ decay in same order, hence in the polynomial order.

The inequality~\eqref{Kull_lieb} can be proved similar to Proposition~4 \cite{Caponnetto}. We obtain the desired results from the inequalities ~\eqref{N.epsilon.k},~\eqref{f.i.j.bound.l},~\eqref{f.i.j.bound.u},~\eqref{sn1},~\eqref{sn2} with~\eqref{Kull_lieb}.
\end{proof}


The following theorem is a restatement of Theorem~3.1 of \cite{DeVore} in non-linear statistical inverse problem setting.

\begin{proposition}\label{err.lower.bound.k}
For any learning algorithm ($\zz\to f_\zz\in\HH_1$) under the hypothesis~$dim(Y)=d<\infty$, Assumption~\ref{A.assumption1} and the condition~\eqref{condition}, there exists a probability measure~$\rho_*\in\mathcal{P}_{\phi,b}$ and~$f_{\rho_*}\in \HH_1$ such that for all~$0<\varepsilon<\varepsilon_\circ$,~$f_\zz$ can be approximated as
$$\mathbb{P}_{\zz\in Z^m}\left\{\|f_{\zz}-f_{\rho_*}\|_{\HH_1} >\varepsilon\upsilon/2\right\} \geq \min\left\{\frac{1}{1+e^{-\ell_\varepsilon/24}}, \vartheta e^{\left(\frac{\ell_\varepsilon}{48}-\frac{\widetilde{C}m\varepsilon^2}{\ell_\varepsilon^b}-\widetilde{C}m\varepsilon^4\right)}\right\}$$
where~$\vartheta=e^{-3/e}$ and~$\ell_\varepsilon=\left\lfloor\frac{1}{2}\left(\frac{\al}{\phi^{-1}(\varepsilon/R)}\right)^{1/b}\right\rfloor$.
\end{proposition}

\begin{proof}
Let~$\varepsilon \leq \varepsilon_0$ and~$f_1,\ldots,f_{N_\varepsilon}$ be as in Proposition~\ref{fi.fj.kull}. Then we define the sets,
$$A_i=\left\{\zz\in Z^m:\|f_{\zz}-f_i\|_{\HH_1}<\frac{\varepsilon\upsilon}{2}\right\}, \text{ for } 1\leq i \leq N_\varepsilon.$$
It is clear from~(\ref{f.i.j.bound.l}) that~$A_i \cap A_j = \emptyset$ if~$i\neq j$. On applying Lemma~3.3 \cite{DeVore} with the probability measures~$\rho_{f_i}^m,~1 \leq i \leq N_\varepsilon$, we obtain that either
\begin{equation}\label{p.either.bound}
p:=\max\limits_{1\leq i \leq N_\varepsilon}\rho_{f_i}^m(A_i^c) \geq \frac{N_{\varepsilon}}{N_\varepsilon+1}
\end{equation}
or
\begin{equation}\label{leibler.l.bound}
\min\limits_{1 \leq j \leq N_\varepsilon} \frac{1}{N_\varepsilon}\sum\limits_{i=1, i \neq j}^{N_\varepsilon}\mathcal{K}(\rho_{f_i}^m,\rho_{f_j}^m) \geq \Psi_{N_\varepsilon}(p),
\end{equation}
where~$\Psi_{N_\varepsilon}(p)=\log(N_\varepsilon)+(1-p)\log\left(\frac{1-p}{p}\right) -p\log\left(\frac{N_\varepsilon-p}{p}\right)$.
Further,
\begin{equation}\label{psi.p.bound}
\Psi_{N_\varepsilon}(p)\geq(1-p)\log(N_\varepsilon)+(1-p)\log(1-p)-\log(p)+2p\log(p) \geq -\log(p)+\log(\sqrt{N_\varepsilon})-3/e.
\end{equation}
Since minimum value of~$x\log(x)$ is~$-1/e$ on~$[0,1]$.

For the joint probability measures~$\rho_{f_i}^m$,~$\rho_{f_j}^m$~$(\rho_{f_i},\rho_{f_j}\in\mathcal{P}_{\phi,b},~1\leq i,j\leq N_\varepsilon)$ from the inequality~(\ref{kullback}) we get,
\begin{equation}\label{leibler.u.bound}
\mathcal{K}(\rho_{f_i}^m,\rho_{f_j}^m)=m\mathcal{K}(\rho_{f_i},\rho_{f_j})\leq \widetilde{C}m\left(\frac{\varepsilon^2}{\ell_\varepsilon^b}+\varepsilon^4\right).
\end{equation}

Therefore the inequalities~(\ref{p.either.bound}),~(\ref{leibler.l.bound}), together with~(\ref{psi.p.bound}) and~(\ref{leibler.u.bound}) implies
$$p:=\max\limits_{1\leq i\leq N_\varepsilon}\left(\mathbb{P}\left\{\zz\in Z^m:\|f_\zz-f_i\|_{\HH_1}>\frac{\upsilon\varepsilon}{2}\right\}\right) \geq\min\left\{\frac{N_\varepsilon}{N_\varepsilon+1},\sqrt{N_\varepsilon}e^{-\frac{3}{e}-\widetilde{C}m\left(\frac{\varepsilon^2}{\ell_\varepsilon^b}+\varepsilon^4\right)}\right\}.$$

From the estimate~(\ref{N.epsilon.k}) for the probability measure~$\rho_*$ such that~$p=\rho_*^m(A_i^c)$ the desired result follows.
\end{proof}

\begin{proof}[Proof of Theorem \ref{err.lower.bound.k.para}]
From Proposition~\ref{err.lower.bound.k} for some probability measure~$\rho^*\in\mathcal{P}_{\phi,b}$ with~$0<\varepsilon<\varepsilon_0$ we get,

\begin{align*}
&\mathbb{P}_{\zz\in Z^m}\left\{\|f_\zz-f_{\rho_*}\|_{\HH_1}> \frac{\upsilon\varepsilon}{2}\right\}\\
\geq &\min\left\{\frac{1}{1+e^{-\ell_\varepsilon/24}},\vartheta e^{-\frac{1}{48}} e^{\left\{\frac{1}{96}\left(\frac{\al}{\phi^{-1}(\varepsilon/R)}\right)^{1/b}-\widetilde{C}m\varepsilon^2\left(\frac{2^{2b-1}\phi^{-1}(\varepsilon/R)}{\al-2^{2b-1}\phi^{-1}(\varepsilon/R)}\right)-\widetilde{C}m\varepsilon^4\right\}}\right\},
\end{align*}
where~$\vartheta=e^{-3/e}$ and~$\ell_\varepsilon=\left\lfloor\frac{1}{2}\left(\frac{\al}{\phi^{-1}(\varepsilon/R)}\right)^{1/b}\right\rfloor$.

Given~$\tau > 0$ for all~$m\in \NN$, we choose~$\varepsilon_m=\tau R\phi\left(\Psi^{-1}\paren{m^{-1/2}}\right)$. Since~$\varepsilon_m$ tends to~$0$ when~$m$
tends to~$+\infty$, therefore for~$m$ large enough~$\varepsilon_m \leq \varepsilon_0$. So Proposition~\ref{err.lower.bound.k} applies ensuring, 
$$\mathbb{P}_{\zz\in Z^m}\left\{\|f_\zz-f_{\rho_*}\|_{\HH_1}> \frac{\tau \upsilon R}{2}\phi\left(\Psi^{-1}\paren{m^{-1/2}}\right)\right\} \geq\min\left\{\frac{1}{1+e^{-\ell_\varepsilon/24}},\vartheta e^{-\frac{1}{48}}e^{c(m)}\right\},$$
where~$c(m)=\left(\Psi^{-1}\paren{m^{-1/2}}\right)^{-1/b}\left\lbrace\frac{\al^{1/b}}{96}-\frac{\widetilde{C}\tau^2R^22^{2b-1}}{\al-2^{2b-1}\Psi^{-1}\paren{m^{-1/2}}}-\widetilde{C}\tau^4R^4\left(\frac{\phi^2\left(\Psi^{-1}\paren{m^{-1/2}}\right)}{\Psi^{-1}\paren{m^{-1/2}}}\right)\right\rbrace$.


Now as~$m$ tends to~$\infty$,~$\varepsilon\to 0$ and~$\ell_\varepsilon\to\infty.$ Therefore, we conclude that
$$\lim\limits_{\tau\to 0}\liminf\limits_{m\to\infty}\inf\limits_{l\in\mathcal{A}}\sup\limits_{\rho\in\mathcal{P}_{\phi,b}} \mathbb{P}_{\zz\in Z^m}\left\{\|f_\zz^l-f_\rho\|_{\HH_1}> \frac{\tau \upsilon R}{2}\phi\left(\Psi^{-1}\paren{m^{-1/2}}\right)\right\}=1.$$
\end{proof}

\begin{proposition}\label{OP_Lip}
Let~$F,~\widetilde{F}:\HH\to\HH$ be the self-adjoint, Hilbert-Schmidt operators over the separable Hilbert space~$\HH$. If~$\theta(t)$ is Lipschitz continuous with Lipschitz constant~$L_\theta\geq 0$, then~$\theta$ is also operator Lipschitz in Hilbert-Schmidt norm:
$$\norm{\theta(F)-\theta(\widetilde{F})}_{HS}\leq L_\theta\norm{F-\widetilde{F}}_{HS}.$$
\end{proposition}
\begin{proof}
Let~$(e_i,\mu_i)_{i\in\NN}$ and~$(\tilde{e}_i,\tilde{\mu}_i)_{i\in\NN}$ be the singular value decompositions of the operators~$F$ and~$\widetilde{F}$, i.e.,~$F=\sum\limits_{i=1}^\infty \mu_i\langle\cdot,e_i\rangle_{\HH}e_i$ and~$\widetilde{F}=\sum\limits_{i=1}^\infty \tilde{\mu}_i\langle\cdot,\tilde{e}_i\rangle_{\HH}\tilde{e}_i$. The values~$(\mu_i)_{i\in\NN}$ and~$(\tilde{\mu}_i)_{i\in\NN}$ are the singular values of the operators~$F$ and~$\widetilde{F}$, respectively. The vectors~$(e_i)_{i\in\NN}$ and~$(\tilde{e}_i)_{i\in\NN}$ are the orthonormal basis of the Hilbert space~$\HH$ and the eigenvectors of the operators~$F$ and~$\widetilde{F}$, respectively. 

We have,
\begin{eqnarray*}
\norm{\{\theta(F)-\theta(\widetilde{F})\}e_j}_{\HH}^2 &=&\norm{\theta(\mu_j)e_j-\sum\limits_{i=1}^\infty \theta(\tilde{\mu}_i)\langle e_j,\tilde{e}_i\rangle_{\HH}\tilde{e}_i}_{\HH}^2=\sum\limits_{i=1}^\infty\left(\theta(\mu_j)-\theta(\mu_i)\right)^2\langle e_j,\tilde{e}_i\rangle_{\HH}^2 \\ \nonumber 
&\leq& L_\theta^2\sum\limits_{i=1}^\infty\left(\mu_j-\tilde{\mu}_i\right)^2\langle e_j,\tilde{e}_i\rangle_{\HH}^2 =L_\theta^2\norm{\{F-\widetilde{F}\}e_j}_{\HH}^2
\end{eqnarray*}
which implies
$$\norm{\theta(F)-\theta(\widetilde{F})}_{HS}\leq L_\theta\norm{F-\widetilde{F}}_{HS}.$$
\end{proof}

\begin{proposition}\label{TB}
Let~$B,~\widetilde{B}:\HH_1\to\HH_2$ be the Hilbert-Schmidt operators over the arbitrary separable Hilbert spaces~$\HH_1$,~$\HH_2$ and~$T=B^*B$,~$\widetilde{T}=\widetilde{B}^*\widetilde{B}$. Then
$$\norm{T^{1/2}-\widetilde{T}^{1/2}}_{HS}\leq\sqrt{2}\norm{B-\widetilde{B}}_{HS}.$$
\end{proposition}
\begin{proof}
Let~$(e_i,f_i,\mu_i)_{i\in\NN}$ and~$(\tilde{e}_i,\tilde{f}_i,\tilde{\mu}_i)_{i\in\NN}$ be the singular value decompositions of the operators~$B$ and~$\widetilde{B}$, i.e.,~$B=\sum\limits_{i=1}^\infty \mu_i\langle\cdot,e_i\rangle_{\HH_1}f_i$ and~$\widetilde{B}=\sum\limits_{i=1}^\infty \tilde{\mu}_i\langle\cdot,\tilde{e}_i\rangle_{\HH_1}\tilde{f}_i$. The values~$(\mu_i)_{i\in\NN}$ and~$(\tilde{\mu}_i)_{i\in\NN}$ are the singular values of the operators~$B$ and~$\widetilde{B}$, respectively. The vectors~$(e_i)_{i\in\NN}$ and~$(\tilde{e}_i)_{i\in\NN}$ are the orthonormal basis of the Hilbert space~$\HH_1$ and the eigenvectors of the operators~$T=B^*B$ and~$\widetilde{T}=\widetilde{B}^*\widetilde{B}$, respectively. The vectors~$(f_i)_{i\in\NN}$ and~$(\tilde{f}_i)_{i\in\NN}$ are the orthonormal basis of the Hilbert space~$\HH_2$ and the eigenvectors of the operators~$BB^*$ and~$\widetilde{B}\widetilde{B}^*$, respectively. 

We have 
\begin{eqnarray}\label{TB1}
\norm{\{T^{1/2}-\widetilde{T}^{1/2}\}e_j}_{\HH_1}^2&=&\langle e_j,Te_j\rangle_{\HH_1}+\langle e_j,\widetilde{T}e_j\rangle_{\HH_1}-2\langle T^{1/2}e_j,\widetilde{T}^{1/2}e_j\rangle_{\HH_1}\\ \nonumber
&=&\norm{Be_j}_{\HH_2}^2+\norm{\widetilde{B}e_j}_{\HH_2}^2-2\mu_j\sum\limits_{i=1}^\infty\tilde{\mu}_i\langle e_j,\tilde{e_i}\rangle_{\HH_1}^2
\end{eqnarray}

and
\begin{eqnarray}\label{TB2}
\norm{\{B-\widetilde{B}\}e_j}_{\HH_2}^2 &=& \norm{Be_j}_{\HH_2}^2+\norm{\widetilde{B}e_j}_{\HH_2}^2-2\langle Be_j,\widetilde{B}e_j\rangle_{\HH_2}\\ \nonumber
&=& \norm{Be_j}_{\HH_2}^2+\norm{\widetilde{B}e_j}_{\HH_2}^2-2\mu_j\sum\limits_{i=1}^\infty \tilde{\mu}_i\langle e_j,\tilde{e_i}\rangle_{\HH_1}\langle f_j,\tilde{f_i}\rangle_{\HH_2}\\ \nonumber
&\geq& \norm{Be_j}_{\HH_2}^2+\norm{\widetilde{B}e_j}_{\HH_2}^2-\mu_j\sum\limits_{i=1}^\infty \tilde{\mu}_i\left(\langle e_j,\tilde{e_i}\rangle_{\HH_1}^2+\langle f_j,\tilde{f_i}\rangle_{\HH_2}^2\right).
\end{eqnarray}

Similarly,

\begin{equation}\label{TB3}
\norm{\{B^*-\widetilde{B}^*\}f_j}_{\HH_1}^2 
\geq \norm{B^*f_j}_{\HH_1}^2+\norm{\widetilde{B}^*f_j}_{\HH_1}^2-\mu_j\sum\limits_{i=1}^\infty \tilde{\mu}_i\left(\langle e_j,\tilde{e_i}\rangle_{\HH_1}^2+\langle f_j,\tilde{f_i}\rangle_{\HH_2}^2\right)
\end{equation}

and
\begin{equation}\label{TB4}
\norm{\{(BB^*)^{1/2}-(\widetilde{B}\widetilde{B}^*)^{1/2}\}f_j}_{\HH_1}^2
=\norm{B^*f_j}_{\HH_1}^2+\norm{\widetilde{B}^*f_j}_{\HH_1}^2-2\mu_j\sum\limits_{i=1}^\infty\tilde{\mu}_i\langle f_j,\tilde{f_i}\rangle_{\HH_2}^2.
\end{equation}

From~\eqref{TB1},~\eqref{TB2},~\eqref{TB3},~\eqref{TB4} we obtain,
$$\norm{\{T^{1/2}-\widetilde{T}^{1/2}\}e_j}_{\HH_1}^2+\norm{\{(BB^*)^{1/2}-(\widetilde{B}\widetilde{B}^*)^{1/2}\}f_j}_{\HH_1}^2\leq \norm{\{B-\widetilde{B}\}e_j}_{\HH_2}^2+\norm{\{B^*-\widetilde{B}^*\}f_j}_{\HH_1}^2$$

which implies 

$$\norm{T^{1/2}-\widetilde{T}^{1/2}}_{HS}^2+\norm{(BB^*)^{1/2}-(\widetilde{B}\widetilde{B}^*)^{1/2}}_{HS}^2\leq \norm{B-\widetilde{B}}_{HS}^2+\norm{B^*-\widetilde{B}^*}_{HS}^2.$$

Hence,
$$\norm{T^{1/2}-\widetilde{T}^{1/2}}_{HS}^2\leq 2\norm{B-\widetilde{B}}_{HS}^2.$$

\end{proof}


\bibliographystyle{amsalpha}
\bibliography{library}

\end{document}